\documentclass[12pt]{amsart}

\usepackage{fullpage}
\usepackage{amsmath}
\usepackage[usenames,dvipsnames]{color}
\usepackage[all]{xy}
\usepackage{pdfsync}
\usepackage{hyperref}
 
\usepackage{graphicx}

\usepackage{etoolbox}
\usepackage{ifthen}
\usepackage{xcolor}
\usepackage[round]{natbib}

\newtoggle{include_figures}
\toggletrue{include_figures}
\numberwithin{equation}{section}

\newtheorem{theorem}{Theorem}[section]
\newtheorem{lemma}[theorem]{Lemma}
\newtheorem{corollary}[theorem]{Corollary}
\newtheorem{proposition}[theorem]{Proposition}
\theoremstyle{definition}
\newtheorem{definition}[theorem]{Definition}

\theoremstyle{remark}
\newtheorem{remark}[theorem]{Remark}
\newtheorem*{remark*}{Remark}

\newcommand{\isom}{\cong}
\newcommand{\R}{\mathbb{R}}
\newcommand{\N}{\mathbb{N}}
\newcommand{\eR}{\overline{\mathbb{R}}}
\newcommand{\norm}[1]{\lVert#1\rVert}
\newcommand{\abs}[1]{\lvert#1\rvert}
\newcommand{\B}{\mathcal{B}}
\newcommand{\xto}{\xrightarrow}
\newcommand{\isomto}{\xrightarrow{\isom}}
\newcommand{\F}{\mathcal{F}}
\newcommand{\st}{\ | \ }
\renewcommand{\L}{\mathcal{L}}
\newcommand{\eps}{\varepsilon}
\newcommand{\incl}{\hookrightarrow}
\renewcommand{\S}{\mathcal{S}}
\renewcommand{\P}{\mathcal{P}}
\newcommand{\A}{\mathcal{A}}

\DeclareMathOperator{\im}{im}
\DeclareMathOperator{\pers}{pers}

\DeclareMathOperator{\Var}{Var}
\DeclareMathOperator{\ess}{ess}

\DeclareMathOperator{\Lip}{Lip}
\DeclareMathOperator{\Pers}{Pers}

\begin{document}

\title[Persistence landscapes]{Statistical topological data analysis using persistence landscapes}

\author{Peter Bubenik}
\address{Department of Mathematics,
       Cleveland State University,
       Cleveland, OH 44115-2214, USA}
\email{peter.bubenik@gmail.com}

\maketitle

\begin{abstract}%   <- trailing '%' for backward compatibility of .sty file
  We define a new topological summary for data that we call the
  persistence landscape.  Since this summary lies in a vector space,
  it is easy to
 % calculate averages of such summaries, and distances
 %  between them
  combine with tools from statistics and machine learning, 
  in contrast to the standard topological summaries.
  Viewed as a random variable with values in a Banach space, this
  summary obeys a strong law of large numbers and a central limit
  theorem.  We show how a number of standard statistical tests can be
  used for statistical inference using this summary.  We also prove
  that this summary is stable and that it can be used to provide lower
  bounds for the bottleneck and Wasserstein distances.
\end{abstract}

% \begin{keywords}
%   topological data analysis, statistical topology, persistent
%   homology, topological summary, persistence landscape
% \end{keywords}

\section{Introduction}
\label{sec:intro}

Topological data analysis (TDA) consists of a growing set of methods
that provide insight to the ``shape'' of data \citep[see the surveys][]{ghrist:survey,carlsson:topologyAndData}. These tools
may be of particular use in understanding global features of high
dimensional data that are not readily accessible using other
techniques. The use of TDA has been limited by the difficulty of
combining the main tool of the subject, the \emph{barcode} or
\emph{persistence diagram} with statistics and machine learning. Here
we present an alternative approach, using a new summary that we call
the \emph{persistence landscape}. The main technical advantage of this
descriptor is that it is a function and so we can use the vector
space structure of its underlying function space. In fact, this
function space is a separable Banach space and we apply the theory of
random variables with values in such spaces.  Furthermore, since the
persistence landscapes are sequences of piecewise-linear functions,
calculations with them are much faster than the corresponding
calculations with barcodes or persistence diagrams, removing a second
serious obstruction to the wider use of topological methods in data
analysis.

Notable successes of TDA include the discovery of a subgroup of breast
cancers by \citet{nlc:topologyBreastCancer}, an understanding of the
topology of the space of natural images by \citet{cidsz:mumford} and the
topology of orthodontic data by \citet{hgk:jasa}, and the detection of
genes with a periodic profile by
\citet{segmentation}. \citet{deSilvaGhrist:homologicalSensorNetworks:capitalized,deSilvaGhrist:coverageInSNvPH:capitalized}
used topology to prove coverage in sensor networks.

In the standard paradigm for TDA, one starts with data that one
encodes as a finite set of points in $\R^n$ or more generally in some
metric space. Then one applies some geometric construction to which
one applies tools from algebraic topology. The end result is a
topological summary of the data. The standard topological descriptors
are the barcode and the persistence diagram 
\citep{elz:tPaS,zomorodianCarlsson:computingPH,cseh:stability}, which give
a multiscale representation of the \emph{homology}
\citep{hatcher:book} of the geometric construction. Roughly, homology
in degree 0 describes the connectedness of the data; homology in
degree 1 detects holes or tunnels; homology in degree 2 captures
voids; and so on. Of particular interest are the homological features
that persist as the resolution changes. We will give precise
definitions and an illustrative example of this method, called
\emph{persistent homology} or \emph{topological persistence}, in
Section~\ref{sec:top}.

Now let us take a statistical view of this paradigm. We consider the
data to be sampled from some underlying abstract probability
space. Composing the constructions above, we consider our topological
summary to be a random variable with values in some summary space
$\S$. In detail, the probability space $(\Omega,\F,\P)$ consists of a
sample space $\Omega$, a $\sigma$-algebra $\F$ of events, and a
probability measure $\P$. Composing our constructions gives a function
$X: (\Omega,\F,\P) \to (\S,\A,\P_*)$, where $\S$ is the summary space,
which we assume has some metric, $\A$ is the corresponding Borel
$\sigma$-algebra, and $\P_*$ is the probability measure on $\S$
obtained by pushing forward $\P$ along $X$. We assume that $X$ is
measurable and thus $X$ is a random variable with values in $\S$.

Here is a list of what we would like to be able to do with
our topological summary.
Let $X_1,\ldots,X_n$ be a sample of independent random variables with
the same distribution as $X$.  
We would like to have a good notion of the mean $\mu$ of $X$ and the
mean $\overline{X}_n$ of the sample; know that $\overline{X}_n$
converges to $\mu$; and be able to calculate $\overline{X}_n(\omega)$,
for $\omega \in \Omega$, efficiently.
We would like to have information the difference
$\overline{X}_n-\mu$, and be able to calculate approximate confidence intervals
related to $\mu$. Given two such samples for random variables $X$ and
$Y$ with values in our summary space, we would like to be able to test
the hypothesis that $\mu_X=\mu_Y$. 
In order to answer these questions we also need an efficient algorithm
for calculating distances between elements of our summary space. 
In this article, we construct a topological summary that we call the
persistence landscape which meets these requirements. 

Our basic idea is to convert the barcode into a function in a somewhat additive manner. The are many possible variations of this construction that may result in more suitable summary statistics for certain applications. Hopefully, the theory presented here will also be helpful in those situations.

We remark that while the persistence landscape has a corresponding
barcode and persistence diagram, the mean persistence landscape does
not.  This is analogous to the situation in which an integer-valued
random variable having a Poisson distribution has a summary statistic, the rate parameter, that is not an integer.

We also remark that the reader may restrict our Banach space results results to the perhaps more familiar Hilbert space setting. However we will need this generality to prove stability of the persistence landscape for, say, functions on the $n$-dimensional sphere where $n>2$.

There has been progress towards combining the persistence diagram and
statistics
\citep{mmh:probability,tmmh:frechet-means,munch:probabilistic-f,Chazal:2013a,Balakrishnan:2013}.
\citet{Blumberg:2012} give a related statistical approach to TDA. 
\citet{giseon:maltose} use the persistence landscape defined here to
study the maltose binding complex and \citet{Chazal:2014} apply the
bootstrap to the persistence landscape.
 The persistence landscape is related to the well
 group defined by \citet{emp:2011}.
% Algorithms for persistence landscapes are given
% in~\citet{bubenikDlotko}.
%The bootstrap has been applied to the persistence landscape~\citet{chazal:2013}.

In Section~\ref{sec:top} we provide the necessary background and
define the persistence landscape and give some of its properties.
In Section~\ref{sec:stat} we introduce the statistical theory of
persistence landscapes, which we apply to a few examples in
Section~\ref{sec:examples}.
In Section~\ref{sec:stability} we prove that the persistence landscape
is stable and that it provides lower bounds for the previously defined
bottleneck and Wasserstein distances.

\section{Topological summaries}
\label{sec:top}

The two standard topological summaries of data are the \emph{barcode}
and the \emph{persistence diagram}.  We will define a new closely-related
summary, the \emph{persistence landscape}, and then
compare it to these two previous summaries. All of these summaries are
derived from the \emph{persistence module}, which we now define.

\subsection{Persistence Modules}
\label{sec:pm}

The main algebraic object of study in topological data analysis is the
persistence module.
A \emph{persistence module} $M$ consists of a vector space
$M_a$ for all $a \in \R$ and linear maps $M(a \leq b): M_a \to M_b$
for all $a \leq b$ such that $M(a \leq a)$ is the identity map and for
all $a \leq b \leq c$, $M(b \leq c) \circ M(a \leq b) = M(a \leq c)$.
%We will assume throughout that all of our vector spaces are finite dimensional.

There are many ways of constructing a persistence module. One example
starts with a set of points $X = \{x_1, \ldots, x_n\}$ in the plane $M
= \R^2$ as shown in the top left of Figure~\ref{fig:balls}. To help
understand this configuration, we ``thicken'' each point, by replacing
each point, $x$, with $B_x(r) = \{y \in M \st d(x,y) \leq r\}$, a disk
of fixed radius, $r$, centered at $x$. The resulting union, $X_r =
\bigcup_{i=1}^n B_r(x_i)$, is shown in Figure~\ref{fig:balls} for various values of
$r$. For each $r$, we can calculate $H(X_r)$, the homology of the
resulting union of disks. To be precise, $H(-)$ denotes
$H_k(-,\mathbb{F})$, the singular homology functor in degree $k$ with
coefficients in a field $\mathbb{F}$.  So $H(X_r)$ is a vector space
that is the quotient of the $k$-cycles modulo those that are
boundaries.  As $r$ increases, the union of disks grows, and the
resulting inclusions induce maps between the corresponding homology
groups. More precisely, if $r \leq s$, the inclusion $\iota_r^s: X_r
\incl X_s$ induces a map $H(\iota_r^s): H(X_r) \to H(X_s)$. The images
of these maps are the \emph{persistent homology} groups. The
collection of vector spaces $H(X_r)$ and linear maps $H(\iota_r^s)$ is
a persistence module.
Note that this construction works for any set of points in $\R^n$ or
more generally in a metric space.

\begin{figure}
\centering
\includegraphics[height=20mm,angle=270]{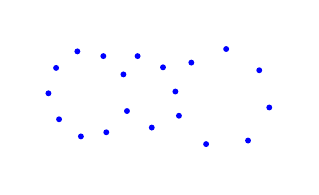}
\includegraphics[height=20mm,angle=270]{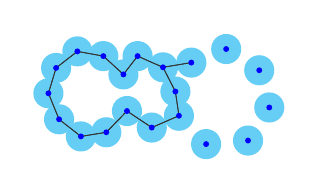}
\includegraphics[height=20mm,angle=270]{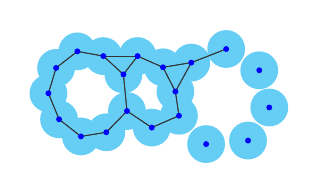}
\includegraphics[height=20mm,angle=270]{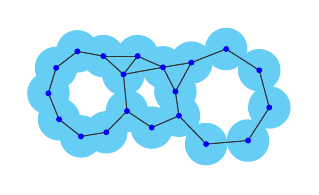}
\includegraphics[height=20mm,angle=270]{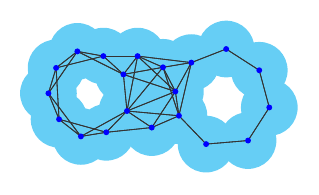}
\includegraphics[height=20mm,angle=270]{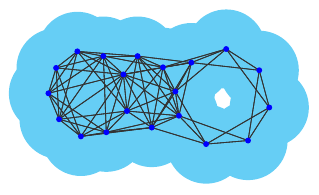}
\includegraphics[height=20mm,angle=270]{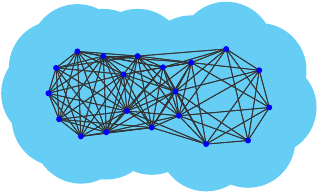}
\caption{A growing union of balls and the 1-skeleton of the
  corresponding \v Cech complex. As the radius grows, features---such
  as connected components and holes---appear and disappear. Here, the
  complexes illustrate the births and deaths of three holes, homology
  classes in degree one. The corresponding birth-death pairs are
  plotted as part of the top left of Figure~\ref{fig:pl}. }
\label{fig:balls}
\end{figure}

The union of balls $X_r$ has a nice combinatorial description.
The \emph{\v Cech complex}, $\check{C}_r(X)$, of the set of balls $\{B_{x_i}(r)\}$ is the
simplicial complex whose vertices are the points $\{x_i\}$ and whose
$k$-simplices correspond to $k+1$ balls with nonempty intersection
(see Figure~\ref{fig:balls}).
This is also called the \emph{nerve}.
It is a basic result that if the ambient space is $\R^n$, $X_r$ is
homotopy equivalent to its \v Cech complex \citep{borsuk:1948}.
So to obtain the singular homology of the union of balls,
one can calculate the simplicial homology of the corresponding \v Cech
complex.
The \v Cech complexes $\{\check{C}_r(X)\}$ together with the
inclusions $\check{C}_r(X) \subseteq \check{C}_s(X)$ for $r \leq s$ form a
filtered simplicial complex.
Applying simplicial homology we obtain a persistence module.
There exist efficient algorithms for calculating the persistent
homology of filtered simplicial complexes \citep{elz:tPaS,mms:zigzag,chen-kerber}.

%\begin{remark}
  The \v Cech complex is often computationally expensive, so many
  variants have been used in computational topology. 
% Using the Delaunay triangulation, one may construct an equivalent subcomplex
%   called the $\alpha$-complex~\citet{???}. 
  A larger, but simpler complex called the Rips complex 
  % is given by the flag complex of the 1-skeleton of the \v Cech
  % complex.
  has as vertices the points ${x_i}$ and has $k$-simplices
  corresponding to $k+1$ balls with all pairwise intersections
  nonempty.  Other possibilities include the witness complexes of
  \citet{deSilvaCarlsson:witness}, graph induced complexes by
  \citet{Dey:2013} and complexes built using kernel density estimators
  and triangulations of the ambient space
  \citep{bckl:nonparametric}. Some of these are used in the examples
  in Section~\ref{sec:examples}.
%\end{remark}

Given any real-valued function $f: S \to \R$ on a topological space
$S$, we can define the associated persistence module, $M(f)$, where
$M(f)(a) = H(f^{-1}((\infty,a]))$ and $M(f)(a \leq b)$ is induced by inclusion.
Taking $f$ to be the the minimum distance to a finite set of points,
$X$, we obtain the first example.

\subsection{Persistence Landscapes}
\label{sec:landscape}

In this section we define a number of functions derived from a
persistence module. Examples of each of these are given in
Figure~\ref{fig:pl}.

Let $M$ be a persistence module.
For $a \leq b$, the corresponding \emph{Betti number} of $M$,
is given by the dimension of the image of the corresponding linear
map. That is,
\begin{equation} \label{eq:rank}
  \beta^{a,b} = \dim( \im (M(a \leq b))).
\end{equation} 

\begin{lemma} \label{lem:rank}
  If $a \leq b \leq c \leq d$ then $\beta^{b,c} \geq \beta^{a,d}$.
\end{lemma}

\begin{proof}
  Since $M(a\leq d) = M(c\leq d) \circ M(b\leq c) \circ M(a\leq b)$,
  this follows from~\eqref{eq:rank}.
\end{proof}

Our simplest function, which we call the
\emph{rank function} is the function $\lambda:
\R^2 \to \R$ given by
\begin{equation*}
  \lambda(b,d) =
  \begin{cases}
    \beta^{b,d} &\text{if } b \leq d\\
    0 &\text{otherwise.}
  \end{cases}
\end{equation*}
% For 
% % the rank function of the degree 1 homology of
% % the example in Figure~\ref{fig:balls} see the top left of 
% an example see Figure~\ref{fig:pl}.

Now let us change coordinates so that the resulting function is supported
on the upper half plane. Let
\begin{equation} \label{eq:change-of-coords}
  m = \frac{b+d}{2}, \quad \text{and} \quad h = \frac{d-b}{2}.
\end{equation}
The \emph{rescaled rank function} is the function $\lambda:
\R^2 \to \R$ given by
\begin{equation*}
  \lambda(m,h) =
  \begin{cases}
    \beta^{m-h,m+h} &\text{if } h \geq 0\\
    0 &\text{otherwise.}
  \end{cases}
\end{equation*}
% For 
% %the rescaled rank function of the degree 1 homology of
% %the example in Figure~\ref{fig:balls} see the top right of 
% an example see Figure~\ref{fig:pl}.
                                   
Much of our theory will apply to these simple functions. However, the
following version, which we will call 
%the \emph{$\N$ persistence  landscape} or just 
the \emph{persistence landscape}, will have some
advantages.

First let us observe that 
%it follows from Lemma~\ref{lem:rank} 
for a fixed $t \in \R$, $\beta^{t-\bullet,t+\bullet}$ is a decreasing
function.  That is,

\begin{lemma} \label{lem:beta-decreasing}
 For $0 \leq h_1 \leq h_2$, 
  \begin{equation*}
    \beta^{t-h_1,t+h_1} \geq \beta^{t-h_2,t+h_2}.
  \end{equation*}
\end{lemma}

\begin{proof}
  Since $t-h_2 \leq t-h_1 \leq t+h_1 \leq t+h_2$, 
% $M(t-h_2 \leq t+h_2) = M(t+h_1 \leq t+h_2) \circ M(t-h_1 \leq t+h_1) \circ M(t-h_2 \leq t-h_1)$. 
by Lemma~\ref{lem:rank}, $\beta^{t-h_2,t+h_2} \leq \beta^{t-h_1,t+h_1}.$
\end{proof}

\begin{definition} \label{def:landscape}
The \emph{persistence landscape} is a function $\lambda: \N \times \R
\to \eR$, where $\eR$ denotes the extended real
numbers, $[-\infty,\infty]$. Alternatively, it may be thought of as a sequence of 
functions $\lambda_k: \R \to \eR$, where $\lambda_k(t) = \lambda(k,t)$.
Define
    \begin{equation*}
        \lambda_k(t) = \sup(m \geq 0 \st \beta^{t-m,t+m} \geq k ).
 \end{equation*}
\end{definition}
% For
% % the persistence landscape of the degree 1 homology of the example
% %in Figure~\ref{fig:balls} see the bottom left of 
% an example see Figure~\ref{fig:pl}.
 
% \begin{definition} \label{def:landscape-from-barcode}
% Given a barcode, we define the corresponding persistence landscape as
% follows,
% \begin{multline*}
% \lambda_k(t) = \max\{ h \st \text{the interval centered at $t$ with radius $h$ is}\\ \text{contained in $k$ intervals in the barcode}\}.
% \end{multline*}
% \end{definition}
% See Figure 2 for a barcode and the corresponding persistence landscape.

The persistence landscape has the following properties.
\begin{lemma} \label{lem:PLproperties}
%  $\lambda_k:\R \to \eR$ satisfies the following properties:
  \begin{enumerate}
    \item $\lambda_k(t) \geq 0$,
    \item \label{it:decreasing} $\lambda_k(t) \geq \lambda_{k+1}(t)$, and
    \item \label{it:lipschitz} $\lambda_k$ is 1-Lipschitz.
  \end{enumerate}
\end{lemma}

The first two properties follow directly from the definition. 
We prove the third in the appendix.

To help visualize the graph of $\lambda: \N \times \R \to \eR$, we can extend it to a function $\overline{\lambda}: \R^2 \to \eR$ by setting
\begin{equation} \label{eq:fbar}
  \overline{\lambda}(x,t) =
  \begin{cases}
    \lambda(\lceil x \rceil,t), & \text{ if } x > 0, \\
    0, & \text{ if } x \leq 0.
  \end{cases}
\end{equation}
% For 
% % this version of the persistence landscape example corresponding to
% % the example in Figure~\ref{fig:balls} see the bottom right of
% an example see
% Figure~\ref{fig:pl}.

We remark that the non-persistent Betti numbers, $\{\dim(M(t))\}$, of a
persistence module $M$ can be read off from the diagonal of the rank
function, the $m$-axis of the rescaled rank function, and from the
support of the persistence landscape.

\begin{figure}
  \centering
\begin{minipage}{60mm}
\includegraphics[width=43mm]{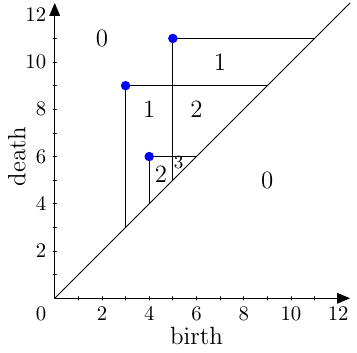} \vspace{1ex}

\includegraphics[width=53mm]{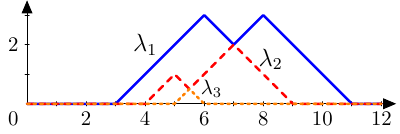}
\end{minipage}
\begin{minipage}{60mm}
\vspace{5ex}

\includegraphics[width=53mm]{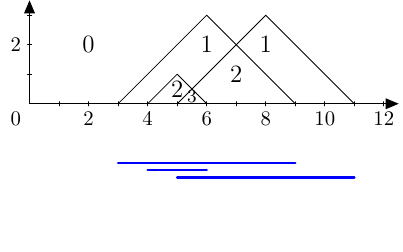} \vspace{-5ex}

\includegraphics[width=53mm]{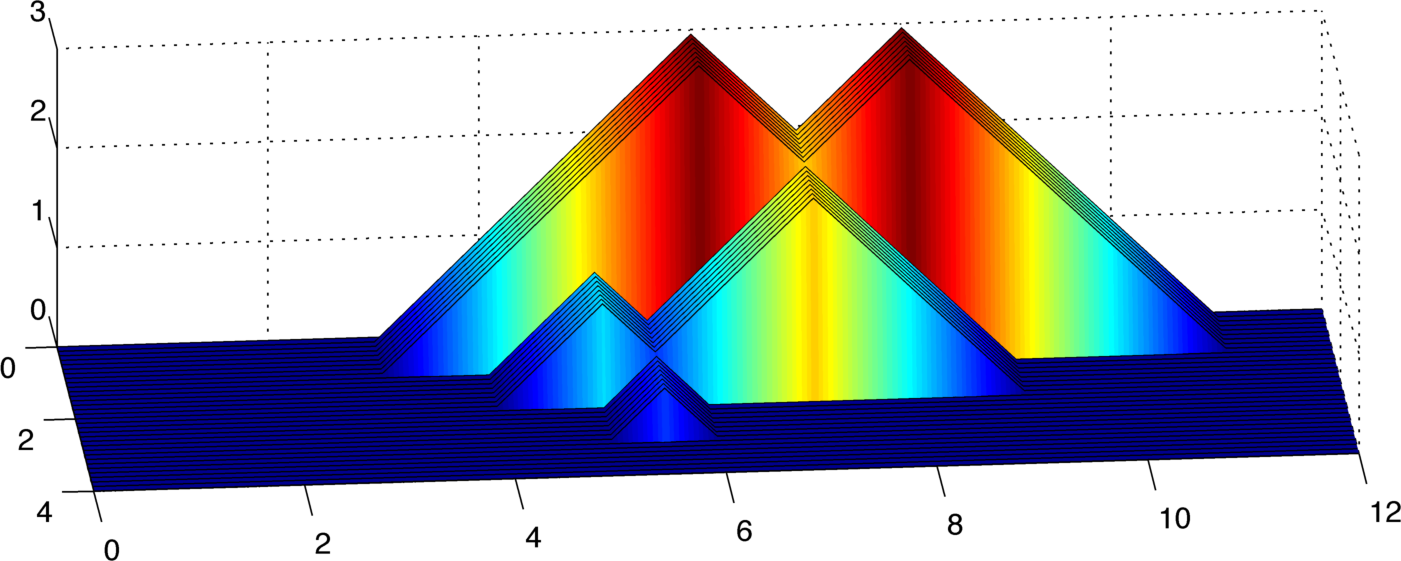}
\end{minipage}
\caption{Persistence landscapes for the homology in degree 1 of the
  example in Figure~\ref{fig:balls}.  For the rank function (top left)
  and rescaled rank function (top right) the values of the functions
  on the corresponding region are given. The top left graph also
  contains the three points of the corresponding persistence diagram. Below
  the top right graph is the corresponding barcode. We also have the
  corresponding persistence landscape (bottom left) and its 3d-version
  (bottom right). Notice that $\lambda_1$ gives a measure of the
  dominant homological feature at each point of the filtration.}
  \label{fig:pl}
\end{figure}

\subsection{Barcodes and Persistence Diagrams} \label{sec:barcode}

All of the information in a (tame) persistence module is completely contained
in a multiset of intervals called a
\emph{barcode} \citep{zomorodianCarlsson:computingPH,crawley-boevey,csgo:persistenceModules}. 
Mapping each interval to its endpoints we obtain the \emph{persistence
  diagram}. 

There exist maps in both directions between these topological
summaries and our functions.
For an example of corresponding persistence diagrams, barcodes and persistence
landscapes, see Figure~\ref{fig:pl}.
Informally, the persistence diagram consists of the ``upper-left
corners'' in our rank function. 
In the other direction, $\lambda(b,d)$ counts the number of points in
the persistence diagram in the upper left quadrant of $(b,d)$.
Informally, the barcode consists of the ``bases of the triangles'' in
the rescaled rank function, and the other direction is
obtained by ``stacking isosceles triangles'' whose bases are the
intervals in the barcode. 
We invite the reader to make the mappings precise.
For example, given a persistence diagram $\{(b_i,d_i)\}_{i=1}^n$,
\begin{equation*}
  \lambda_k(t) = k\text{th largest value of } \min(t-b_i,d_i-t)_+,
\end{equation*}
where $c_+$ denotes $\max(c,0)$.
The fact that barcodes are a complete invariant of persistence
modules is central to these equivalences.

The geometry of the space of persistence diagrams makes it hard to
work with. For example, sets of persistence diagrams need not have a
unique (Fr\'echet) mean \citep{mmh:probability}.  In contrast, the
space of persistence landscapes is very nice. So a set of persistence
landscapes has a unique mean~\eqref{eq:mean}. See
Figure~\ref{fig:mean}.

\begin{figure}
  \centering
\includegraphics[width=5cm]{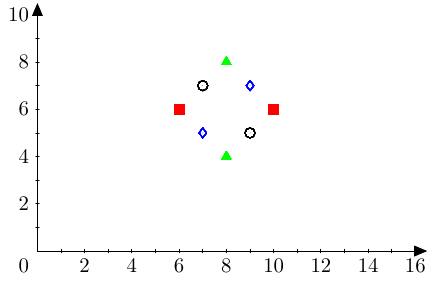}
\includegraphics[width=5cm]{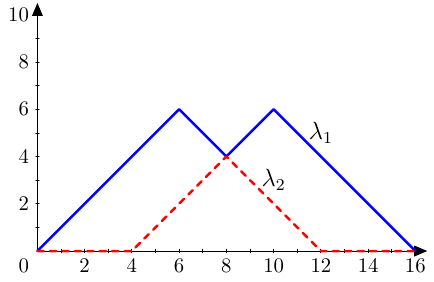}\\
\includegraphics[width=5cm]{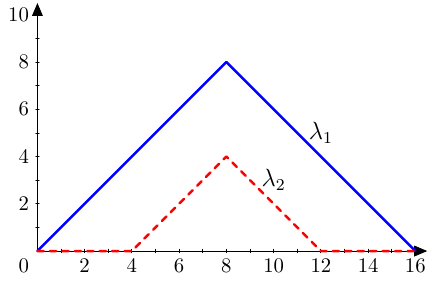}
\includegraphics[width=5cm]{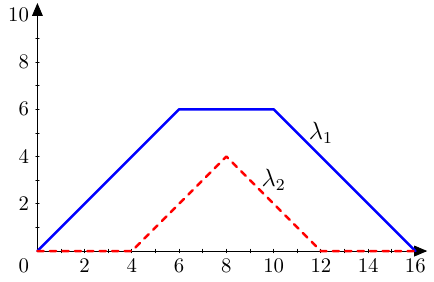}
\caption{Means of persistence diagrams and persistence landscapes. Top
  left: the rescaled persistence diagrams $\{(6,6),(10,6)\}$ and
  $\{(8,4),(8,8)\}$ have two (Fr\'echet) means: $\{(7,5),(9,7)\}$ and
  $\{(7,7),(9,5)\}$. In contrast their corresponding persistence
  landscapes (top right and bottom left) have a unique mean (bottom
  right).}
\label{fig:mean}
\end{figure}

%\begin{remark}
  Compared to the persistence diagram, the barcode has extra
  information on whether or not the endpoints of the intervals are
  included. This finer information is seen in the rank function and
  rescaled rank function, but not in the persistence
  landscape. However when we pass to the corresponding $L^p$ space in
  Section~\ref{sec:norm}, this information disappears.
%\end{remark}

\subsection{Norms for Persistence Landscapes} \label{sec:norm}

Recall that for a measure space $(\S,\mathcal{A},\mu)$, and a function
$f:\S \to \R$ defined $\mu$-almost everywhere, 
for $1 \leq p < \infty$,
$\norm{f}_p = \left[\int \abs{f}^p d\mu \right]^{\frac{1}{p}}$, 
and $\norm{f}_{\infty} = \ess \sup f  = \inf \{a \st \mu\{s \in \S \st f(s) > a\} = 0\}$.
For $1 \leq p \leq \infty$, $\L^p(\S) = \{f: \S \to \R \st \norm{f}_p
< \infty\}$ and define $L^p(\S) = \L^p(\S)/\sim$, where $f \sim g$ if $\norm{f-g}_p=0$.

On $\R$ and $\R^2$ we will use the Lebesgue measure.
On $\N \times \R$, we use the product of the counting measure on $\N$ and the Lebesgue measure on $\R$.
For $1 \leq p < \infty$ and $\lambda: \N \times \R \to \eR$, 
\begin{equation*}
\norm{\lambda}_p^p = \sum_{k=1}^{\infty} \norm{\lambda_k}_p^p,
\end{equation*}
where $\lambda_k(t) = \lambda(k,t)$.  By
Lemma~\ref{lem:PLproperties}\eqref{it:decreasing},
$\norm{\lambda}_{\infty} = \norm{\lambda_1}_{\infty}$.  If we extend
$f$ to $\overline{\lambda}: \R^2 \to \eR$, as in \eqref{eq:fbar}, we
have $\norm{\lambda}_p = \norm{\overline{\lambda}}_p$, for $1 \leq p
\leq \infty$.

If $\lambda$ is any of our functions corresponding to a barcode
that is a finite collection of finite intervals, then
$\lambda \in \L^p(\S)$
for $1 \leq p \leq \infty$, where $\S$ equals $\N \times \R$ or $\R^2$.

Let $\lambda_{bd}$ and $\lambda_{mh}$ denote the rank function
and the rescaled rank function corresponding to a persistence
landscape $\lambda$, and let $D$ be the corresponding persistence
diagram.  
Let $\pers_2(D)$ denote the sum of the squares of the lengths of the
intervals in the corresponding barcode, and let $\pers_{\infty}(D)$ be
the length of the longest interval.

\begin{proposition}
  \begin{enumerate}
  \item \label{it:lambda1} $\norm{\lambda}_1 = \norm{\lambda_{mh}}_1 =
    \frac{1}{2} \norm{\lambda_{bd}}_1 
%= W_2(D,\emptyset)^2 
= \frac{1}{4}\pers_2(D)$, and
  \item \label{it:lambda2} $\norm{\lambda}_{\infty} =
    \norm{\lambda_1}_{\infty} 
% = W_{\infty}(D,\emptyset)
= \frac{1}{2} \pers_{\infty}(D)$.
  \end{enumerate}
\end{proposition}

  \begin{proof}
    \begin{enumerate}
    \item To see that $\norm{\lambda}_1 = \norm{\lambda_{mh}}$ we
      remark that both are the volume of the same solid. The change of
      coordinates
%~\eqref{eq:change-of-coords} 
      implies that $\norm{\lambda_{mh}}_1 =
      \frac{1}{2}\norm{\lambda_{bd}}_1$.  If $D = \{(b_i,d_i)\}$, then
      each point $(b_i,d_i)$ contributes $h_i^2$ to the volume
      $\norm{\lambda_{mh}}_1$, where $h_i=\frac{d_i-b_i}{2}$.  So
      $\norm{\lambda_{mh}}_1 = \sum_i h_i^2%
% = W_2(D,\emptyset)^2
$. Finally, $\pers_2(D) = \sum_i
      (2h_i)^2 = 4\sum_ih_i^2$.
    \item Lemma~\ref{lem:PLproperties}\eqref{it:decreasing} implies
      that $\norm{\lambda}_{\infty} = \norm{\lambda_1}_{\infty}$. If
      $D = \{(b_i,d_i)\}$, then %both
$\norm{\lambda}_{\infty} = 
\sup_i \frac{d_i-b_i}{2}$. %\qedhere
    \end{enumerate}
  \end{proof}

  We remark that the quantities in \ref{it:lambda1} and
  \ref{it:lambda2} also equal $W_2(D,\emptyset)^2$ and
  $W_{\infty}(D,\emptyset)$ respectively (see
  Section~\ref{sec:stability} for the corresponding definitions).

\section{Statistics with landscapes}
\label{sec:stat}

Now let us take a probabilistic viewpoint.  First, we assume that our
persistence landscapes lie in $L^p(\S)$ for some $1 \leq p < \infty$,
where $\S$ equals $\N \times \R$ or $\R^2$.  In this case, $L^p(\S)$
is a separable Banach space.  When $p=2$ we have a Hilbert space;
however, we will not use this structure.  In some
examples, the persistence landscapes will only be stable for some
$p>2$ (see Theorem~\ref{thm:landscape-stability}).

\subsection{Landscapes as Banach Space Valued Random Variables}
\label{sec:banach}

Let $X$ be a random variable on some underlying
probability space $(\Omega, \F, P)$, with corresponding 
persistence landscape $\Lambda$, a Borel random variable with
values in the separable Banach space $L^p(\S)$.
That is, for $\omega \in \Omega$, $X(\omega)$ is the data and
$\Lambda(\omega) = \lambda(X(\omega)) =: \lambda$ is the corresponding
topological summary statistic.

Now let $X_1, \ldots, X_n$ be independent and identically distributed
copies of $X$, and let $\Lambda^1,\ldots,\Lambda^n$ be the
corresponding persistence landscapes.  Using the vector space
structure of $L^p(\S)$, the \emph{mean landscape}
$\overline{\Lambda}^n$ is given by the pointwise mean.
That is, $\overline{\Lambda}^n(\omega) = \overline{\lambda}^n$, where
  \begin{equation} \label{eq:mean}
    \overline{\lambda}^n(k,t) = \dfrac{1}{n} \sum_{i=1}^n \lambda^i(k,t).
  \end{equation}
Let us interpret the mean landscape.
If $B_1,\ldots,B_n$ are the barcodes corresponding to the persistence landscapes
$\lambda^1,\ldots,\lambda^n$,
then for $k\in \N$ and $t \in \R$, $\overline{\lambda}^n(k,t)$ is the average value of the largest radius interval centered at $t$ that is contained in $k$ intervals in the barcodes $B_1,\ldots,B_n$.

For those used to working with persistence diagrams, it
is tempting to try to find a persistence diagram whose persistence
landscape is closest to a given mean landscape. While this is an
interesting mathematical question, we would like to suggest that the more
important practical issue is using the mean landscape to understand the data.

We would like to be able to say that the mean landscape converges to
the expected persistence landscape. To say this precisely we need some
notions from probability in Banach spaces.

\subsection{Probability in Banach Spaces}
\label{sec:probability}

Here we present some results from probability in Banach spaces.  For a
more detailed exposition we refer the reader
to \citet{ledoux-talagrand:book}.

Let $\B$ be a real separable Banach space with norm $\norm{\cdot}$.
Let $(\Omega,\F,P)$ be a probability space, and let $V: (\Omega,\F,P)
\to \B$ be a Borel random variable with values in $\B$.
The composite $\norm{V}: \Omega \xto{V} \B \xto{\norm{\cdot}}
\R$ is a real-valued random variable.
Let $\B^*$ denote the topological dual space of continuous linear
real-valued functions on $\B$.
For $f \in \B^*$, the composite $f(V): \Omega \xto{V} \B \xto{f} \R$
is a real-valued random variable.

For a real-valued random variable $Y:(\Omega,\F,P) \to \R$, the
\emph{mean} or  \emph{expected value}, is given by $E(Y) = \int Y \ dP
= \int_{\Omega} Y(\omega) \ dP(\omega)$.
%
% We say that $V$ is Bochner integrable or strongly integrable, if $E
% \norm{V} < \infty$. 
%
We call an element $E(V) \in \B$ the \emph{Pettis integral} of $V$ if
$E(f(V)) = f(E(V))$ for all $f \in \B^*$.

\begin{proposition}
  If $E\norm{V} < \infty$, then $V$ has a Pettis integral and
  $\norm{E(V)} \leq E\norm{V}$.
\end{proposition}

Now let $(V_n)_{n \in \N}$ be a sequence of independent copies of $V$.
For each $n \geq 1$, let  $S_n = V_1 + \cdots + V_n$.
For a sequence $(Y_n)$ of $\B$-valued random variables, we say that 
$(Y_n)$ \emph{converges almost surely} to a $\B$-valued random
variable $Y$, if 
%$P(\omega \in \Omega \st \lim_{n\to \infty} Y_n(\omega) = Y(\omega)) = 1$.
$P(\lim_{n\to \infty} Y_n = Y) = 1$.

\begin{theorem}[Strong Law of Large Numbers] \label{thm:slln-banach}
  $(\frac{1}{n}S_n) \to E(V)$ almost surely if and
  only if $E\norm{V} < \infty$.
\end{theorem}

For a sequence $(Y_n)$ of $\B$-valued random variables, we say that 
$(Y_n)$ \emph{converges weakly} to a $\B$-valued random
variable $Y$, if 
$\lim_{n\to\infty} E(\varphi(Y_n)) = E(\varphi(Y))$ for all bounded
continuous functions $\varphi: \B \to \R$.
A random variable $G$ with values in $\B$ is said to be
\emph{Gaussian} if for each $f \in \B^*$, $f(G)$ is a real valued
Gaussian random variable with mean zero.  
The \emph{covariance structure} of a $\B$-valued random variable, $V$,
is given by the expectations $E[(f(V)-E(f(V)))(g(V)-E(g(V)))]$, where
$f,g \in B^*$.  
A Gaussian random variable is determined by its covariance structure.
From \citet{hjp:lln-clt-banach} we have the following.

\begin{theorem}[Central Limit Theorem]%, \citet{hjp:lln-clt-banach}]
\label{thm:clt-banach}

  Assume that $\B$ has type 2. (For example $\B = L^p(\S)$, with $2
  \leq p < \infty$.) If $E(V) = 0$ and $E(\norm{V}^2) < \infty$ then
  $\frac{1}{\sqrt{n}}S_n$ converges weakly to a Gaussian random
  variable $G(V)$ with the same covariance structure as $V$.
\end{theorem}

\subsection{Convergence of Persistence Landscapes}
\label{sec:main-thms}

Now we will apply the results of the previous section to persistence landscapes.

Theorem~\ref{thm:slln-banach} directly implies the following.

\begin{theorem}[Strong Law of Large Numbers for persistence
  landscapes]
\label{thm:slln} \hspace*{1em}\\
  $\overline{\Lambda}^n \to E(\Lambda)$ almost surely if and
  only if $E\norm{\Lambda} < \infty$.
\end{theorem}

\begin{theorem}[Central Limit Theorem for peristence landscapes] 
\label{thm:clt}
  Assume $p\geq 2$.
  If $E\norm{\Lambda} < \infty$ and $E(\norm{\Lambda}^2) < \infty$ then
  $\sqrt{n}[\overline{\Lambda}^n - E(\Lambda)]$ converges weakly
  to a Gaussian random variable with the same covariance
  structure as $\Lambda$.
\end{theorem}

\begin{proof}
  Apply Theorem~\ref{thm:clt-banach} to $V=\lambda(X)-E(\lambda(X))$.
\end{proof}

Next we apply a functional to the persistence landscapes to obtain a
real-valued random variable that satisfies the usual central limit theorem.

\begin{corollary} \label{cor:functional} 
  Assume $p\geq 2$,
  $E\norm{\Lambda}<\infty$ and $E(\norm{\Lambda}^2)<\infty$.  For any
  $f \in L^q(\S)$ with $\frac{1}{p} + \frac{1}{q} = 1$, let
\begin{equation} \label{eq:Y}
  Y = \int_{\S} f \Lambda = \norm{f\Lambda}_1.
  % \text{ and } Y_i = \int_{\S} f \Lambda^i.
\end{equation}
Then  
\begin{equation} \label{eq:clt}
  \sqrt{n}[\overline{Y}_n - E(Y)] \xto{d} N(0,\Var(Y)).
\end{equation}
where $d$ denotes convergence in distribution and $N(\mu,\sigma^2)$ is
the normal distribution with mean $\mu$ and variance $\sigma^2$.
\end{corollary}

\begin{proof}
  Since $V=\Lambda-E(\Lambda)$ satisfies the central limit
  theorem in $L^p(\S)$, for any $g \in L^p(\S)^*$, the real random
  variable $g(V)$ satisfies the central limit theorem in $\R$ with
  limiting Gaussian law with mean 0 and variance $E(g(V)^2)$. If we take $g(h) =
  \int_{\S}  fh$, where $f \in L^q(\S)$, with
  $\frac{1}{p}+\frac{1}{q}=1$, then $g(V)=Y-E(Y)$ and $E(g(V)^2) = \Var(Y)$.
\end{proof}

\subsection{Confidence Intervals}
\label{sec:confidence-interval}

The results of Section~\ref{sec:main-thms} allow us to obtain
approximate confidence intervals for the expected values of functionals on
persistence landscapes.

Assume that $\lambda(X)$ satisfies the conditions of
Corollary~\ref{cor:functional} and that $Y$ is a corresponding real
random variable as defined in~\eqref{eq:Y}.  By
Corollary~\ref{cor:functional} and Slutsky's theorem we may use
the normal distribution to obtain the approximate $(1-\alpha)$ confidence interval for
$E(Y)$ using
\begin{equation*}
  \overline{Y}_n \pm z^* \frac{S_n}{\sqrt{n}},
\end{equation*}
where $S_n^2 = \frac{1}{n-1}\sum_{i=1}^n(Y_i-\overline{Y}_n)^2$, and
$z^*$ is the upper $\frac{\alpha}{2}$ critical
value for the normal distribution.

\subsection{Statistical Inference using Landscapes I}
\label{sec:inference}

Here we apply the results of Section~\ref{sec:main-thms} to
hypothesis testing using persistence landscapes.

Let $X_1,\ldots,X_n$ be an iid copies of the random variable $X$ and
let $X'_1,\ldots,X'_{n'}$ be an iid copies of the random variable
$X'$.
Assume that the corresponding persistence landscapes
$\Lambda$, $\Lambda'$ lie in $L^p(\S)$, where $p\geq 2$. %$2 \leq p < \infty$.
Let $f \in L^q(\S)$, where $\frac{1}{p}+\frac{1}{q}=1$.
Let $Y$ and $Y'$ be defined as in~\eqref{eq:Y}.
Let $\mu = E(Y)$ and $\mu'=E(Y')$.
We will test the null hypothesis that $\mu=\mu'$.
First we recall that the sample mean $\overline{Y} =
\frac{1}{n}\sum_{i=1}^n Y_i$ is an unbiased estimator of $\mu$ and the
sample variance $s_Y^2 =
\frac{1}{n-1}\sum_{i=1}^n(Y_i-\overline{Y})^2$ is an unbiased
estimator of $\Var(Y)$ and similarly for $\overline{Y'}$ and $s_{Y'}^2$.
By Corollary~\ref{cor:functional}, $Y$ and $Y'$ are asymptotically normal.

We use
the two-sample z-test.
Let 
\begin{equation*}
  z = \frac{ \overline{Y} - \overline{Y'} }{
    \sqrt{\frac{S_Y^2}{n}+\frac{S_{Y'}^2}{n'}}},
\end{equation*}
where the denominator is the standard error for the difference.
From this standard score a p-value may be obtained from the normal distribution.

\subsection{Choosing a Functional}
\label{sec:functional}

To apply the above results, one needs to choose a functional, $f \in
L^q(\S)$. This choice will need to be made with an understanding of
the data at hand. Here we present a couple of options.

If each $\lambda = \Lambda(\omega)$ is supported by
$\{1,\ldots,K\}\times [-B,B]$,
take
\begin{equation} \label{eq:f}
  f(k,t) =
  \begin{cases}
    1 &\text{ if }t \in [-B,B] \text{ and } k \leq K\\
    0 &\text{ otherwise.}
  \end{cases}
\end{equation}
Then
$\norm{f\Lambda}_1 = \norm{\Lambda}_1$.

If the parameter values for which the
persistence landscape is nonzero are bounded by $\pm B$, then
we have a nice choice of functional for the persistence landscape
that is unavailable for the (rescaled) rank function. We can choose a
functional that is sensitive of the first $K$ dominant homological
features.  That is, using $f$ in \eqref{eq:f},
$\norm{f\lambda}_1 = \sum_{k=1}^K \norm{\Lambda_k}_1$.
Under this weaker assumption we can also take $f_k(t) = \frac{1}{k^r}
  \chi_{[-B,B]}$, where $r >1$. Then $\norm{f\Lambda}_1 =
  \sum_{k=1}^{\infty} \frac{1}{k^r}\norm{\Lambda_k(t)}_1$.

% One way to enforce that the parameter values for which the persistence
% landscape is nonzero are bounded is to threshold the underlying
% persistence module $M$ as follows. For some $B > 0$, let $M_B$ be the
% persistence module that is equal to $M$ on $[-B,B]$ but equal to zero
% outside this interval. Let $\lambda_B(M) = \lambda(M_B)$ and let
% $D_B(M) = D(M_B)$, where $D(M)$ denotes the persistence diagram of $M$.

The condition that $\lambda$ is supported by $\N \times [-B,B]$ can
often be enforced by using reduced
homology or by applying extended 
persistence \citep{cseh:extendingP,bubenikScott:1}
or by simply truncating the intervals in the corresponding barcode
at some fixed values. 
We remark that certain experimental data may have bounds on the number
of intervals.  For example, in the protein data considered using the
ideas presented here in \citet{giseon:maltose}, the simplicial
complexes have a fixed number of vertices.

\subsection{Statistical Inference using Landscapes II}
\label{sec:inference2}

The functionals suggested in
Section~\ref{sec:functional} in the hypothesis test given in
Section~\ref{sec:inference} may not have enough power to
discriminate between two groups with different persistence in some
examples.

To increase the power, one can apply a vector of functionals and then apply
Hotelling's $T^2$ test. For example, consider $Y =
(\int(\Lambda_1-\Lambda'_1),\ldots,\int(\Lambda_K-\Lambda'_K))$, where
$K\ll n_1+n_2-2$.

This alternative will not be sufficient if the persistence landscapes
are translates of each other, 
(see Figure~\ref{fig:noisy}).
An additional approach is to compute the distance between
the mean landscapes of the two groups and obtain a p-value using a
permutation test. This is done in 
the Section~\ref{sec:torus-sphere}. 
%For another example, see~\citet{bubenikDlotko}.
This test has been applied to persistence diagrams and
barcodes \citep{cbk:ipmi2009,robinsonTurner:2013}.

\section{Examples}
\label{sec:examples}

The persistent homologies in this section were calculated using
javaPlex \citep{javaplex} and Perseus by \citet{perseus}. Another publicly
available alternative is Dionysus by \citet{dionysus}.
In Section~\ref{sec:grf} we use Matlab code courtesy of Eliran Subag that
implements an algorithm from \citet{wood-chan:simulation}.

\subsection{Linked Annuli}
\label{sec:annuli}

\begin{figure}
\begin{center}
\includegraphics{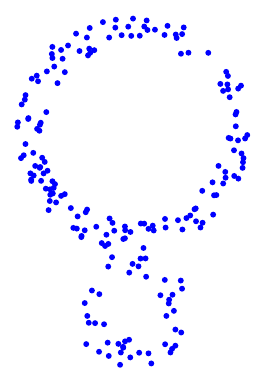}
\includegraphics{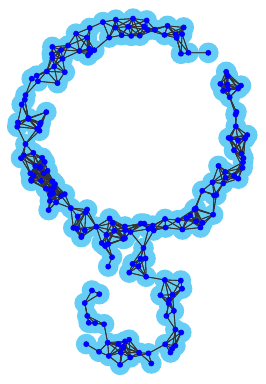}

\includegraphics{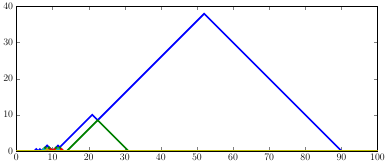}
\includegraphics{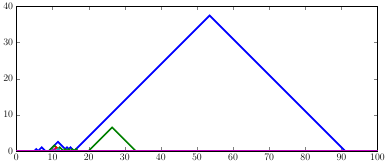}

\includegraphics{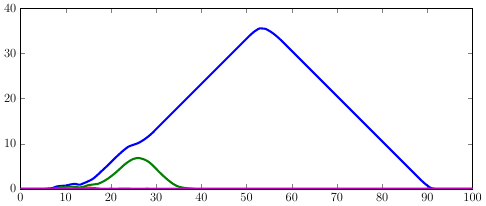}
\end{center}

\caption{200 points were sampled from a pair of linked annuli. Here we
  show the points and a corresponding union of balls and
  1-skeleton of the \v Cech complex. This was repeated 100 times. Next
  we show two of the degree one persistence landscapes and the mean
  degree one persistence landscape.}
\label{fig:annuli}
\end{figure}

We start with a simple example to illustrate the techniques.
Following \citet{munch:probabilistic-f}, we sample 200
points from the uniform distribution on the union of two annuli. We
then calculate the corresponding persistence landscape in degree one
using the Vietoris-Rips complex. We repeat this 100 times and
calculate the mean persistence landscape. See
Figure~\ref{fig:annuli}.

Note that in the degree one barcode of this example, it is very likely
that there will be one large interval, one smaller interval born at
around the same time, and all other intervals are smaller and die
around the time the larger two intervals are born.

\subsection{Random geometric complexes}
\label{sec:points}

The (non-persistent) homology of random geometric complexes has been
studied in \cite{kahle:randomGeometricComplexes,Kahle-Meckes:2013,bobrowskiAdler:distanceFunctions}.

We sample 100 points from the uniform distribution on the unit cube
$[0,1]^3$, and calculate the persistence landscapes in degrees 0, 1
and 2 of the corresponding Vietoris-Rips complex. In degree 0, we use
reduced homology. We repeat this 1000
times and calculate the corresponding mean persistence
landscapes. See Figure~\ref{fig:points}.

Since the number of simplices is bounded and the filtration is
bounded, these persistence landscapes have finite support.  As
discussed in Section~\ref{sec:functional}, we can choose the
functional given by the indicator function on this support.  We obtain
the real random variable $Y = \norm{\lambda(X)}_1 = \frac{1}{4}
\pers_2(D(X))$, where $D(X)$ is the persistence diagram corresponding
to $\lambda(X)$.  Following Section~\ref{sec:confidence-interval} we
calculate the approximate 95\% confidence intervals of $E(Y)$ in
degrees 0, 1 and 2 to be $[0.1534,0.1545]$, $[0.0064,0.0066]$ and
$[0.0002,0.0003]$.

\begin{remark}
The graphs in Figure~\ref{fig:points} may be thought of as a
persistent homology version of the graph in Figure 2 of~\cite{Kahle-Meckes:2013}.
\end{remark}

\begin{figure}
  \centering

\includegraphics[width=5cm]{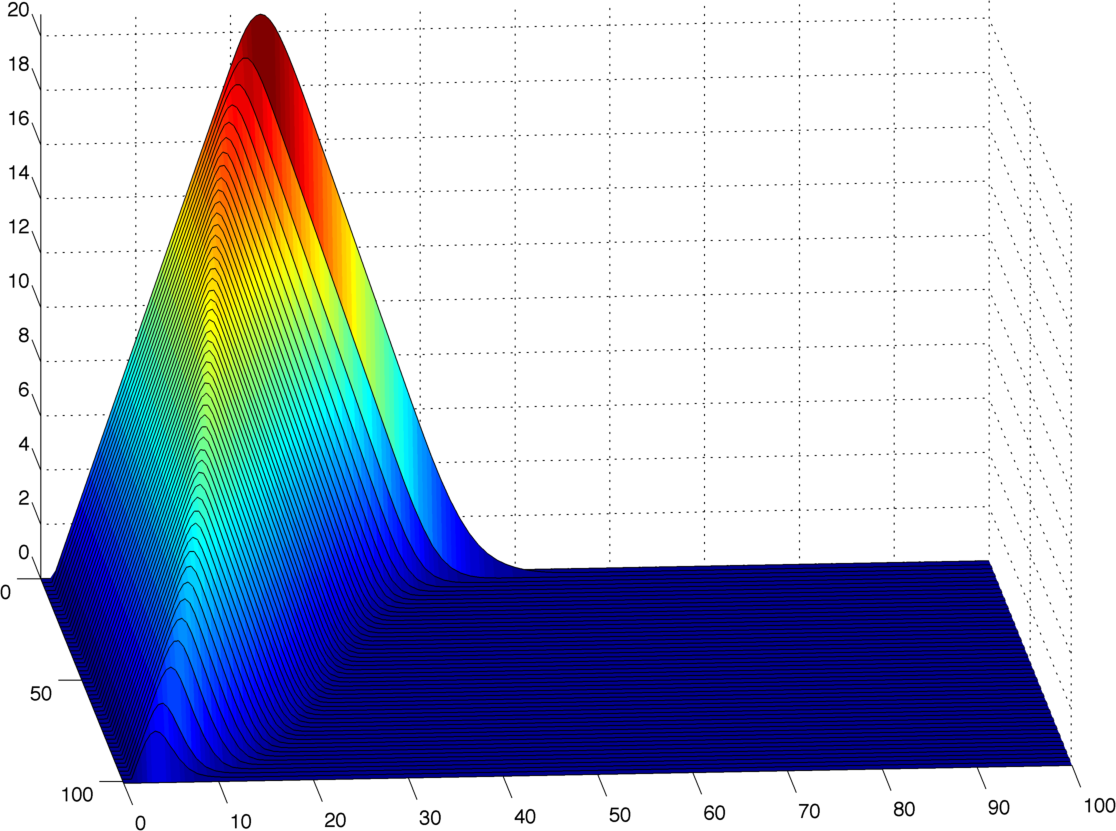}
\includegraphics[width=5cm]{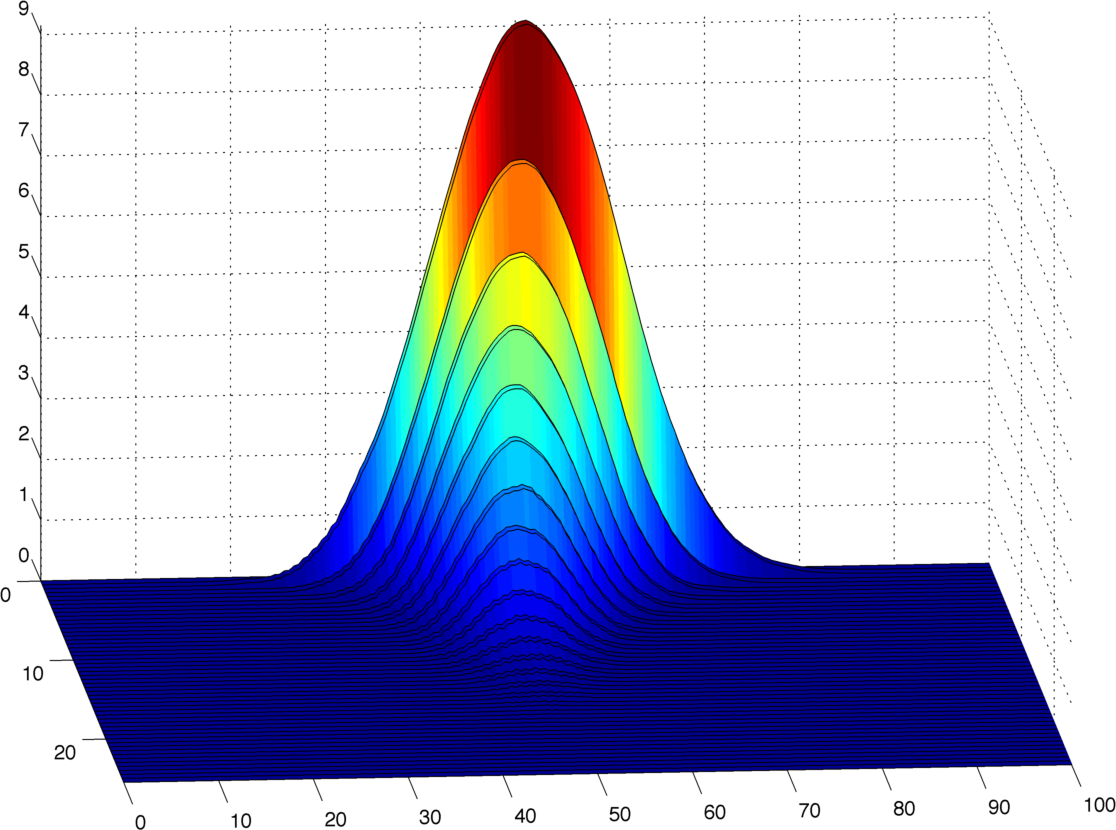}

\includegraphics[width=5cm]{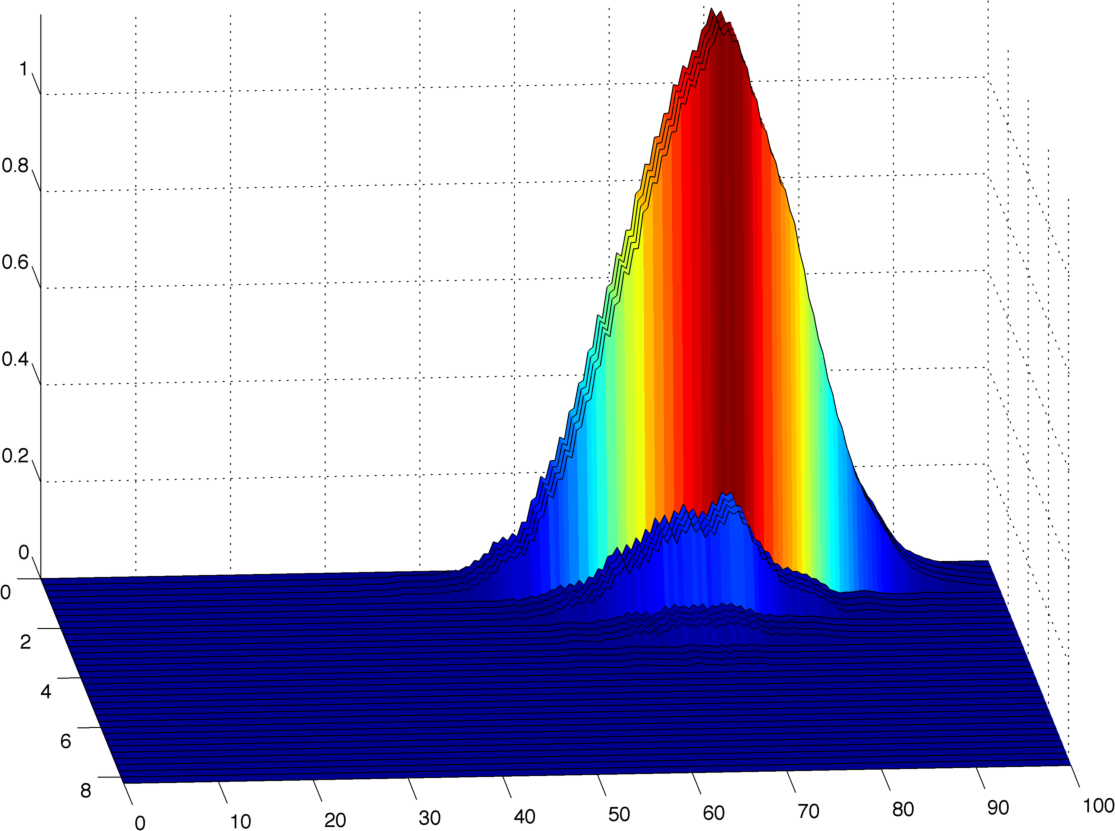}
\includegraphics[width=5cm]{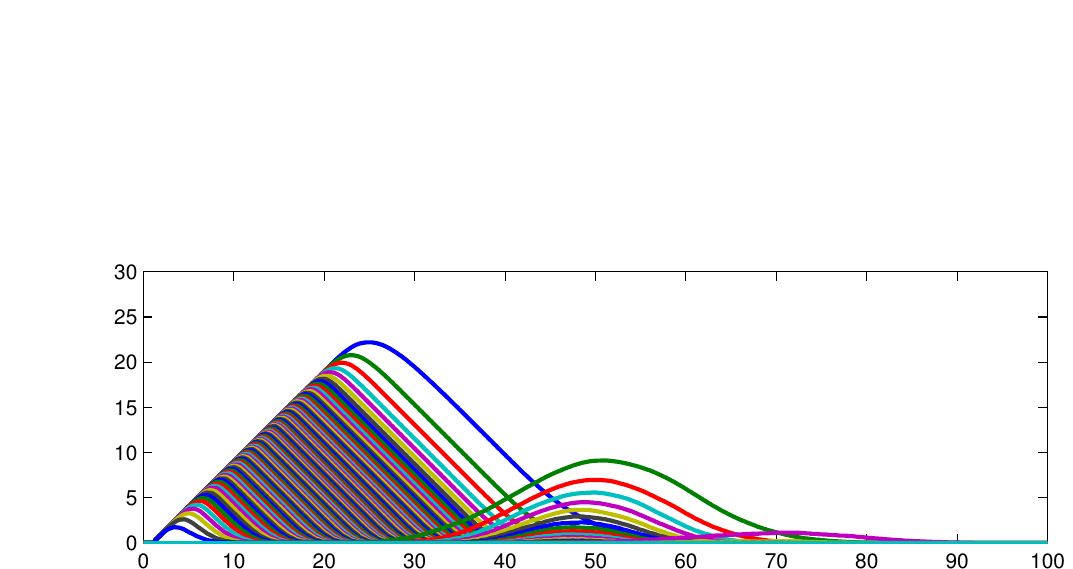}

\caption{The mean persistence landscapes of 1000 Vietoris-Rips
  complexes of 100 points sampled uniformly from the cube
  $[0,1]^3$. The top left, top right and bottom left are homology in
  degrees 0, 1 and 2. The bottom right is the superposition of all
  three. Note that the filtration parameter has been rescaled, so 100
  in the width and height of the graphs equals 0.3.}

  \label{fig:points}
\end{figure}

\subsection{Erd\"os-R\'enyi random clique complexes}
\label{sec:clique}

The (non-persistent) homology of the random complexes in this section
has been studied~\cite{Kahle-Meckes:2013}.

Let $G(n)$ be the following random filtered graph. There are $n$
vertices with filtration value $0$. Each of the possible $n \choose 2$
edges has a filtration value which is chosen independently from a uniform
distribution on $[0,1]$.
Let $X(n)$ be the clique complex of $G(n)$.
In Figure~\ref{fig:clique} we show the mean persistence landscapes
of a sample of 10 independent copies of $G(100)$ in degrees 0, 1, 2,
and 3, where in degree 0 we use reduced homology.
For computational reasons, we only considered the subcomplex of
$G(100)$ with filtration values at most 0.55.

The graphs in Figure~\ref{fig:clique} are a persistent homology
version of Figure 1 in~\cite{Kahle-Meckes:2013}. 
In fact the latter graphs
are given by the support of the graphs in Figure~\ref{fig:clique}.

As in the example in Section~\ref{sec:points}, we let $Y =
\norm{\lambda(X)}_1 = \frac{1}{4}\pers_2(D(X))$. The approximate 95\%
confidence intervals of $E(Y)$ in degrees 0, 1, 2 and 3 are estimated
to be $[0.0034,0.0039]$, $[0.751,0.777]$, $[1.971,2.041]$ and
$[2.591,2.618]$.

\begin{figure}
  \centering
  
\includegraphics[width=5cm]{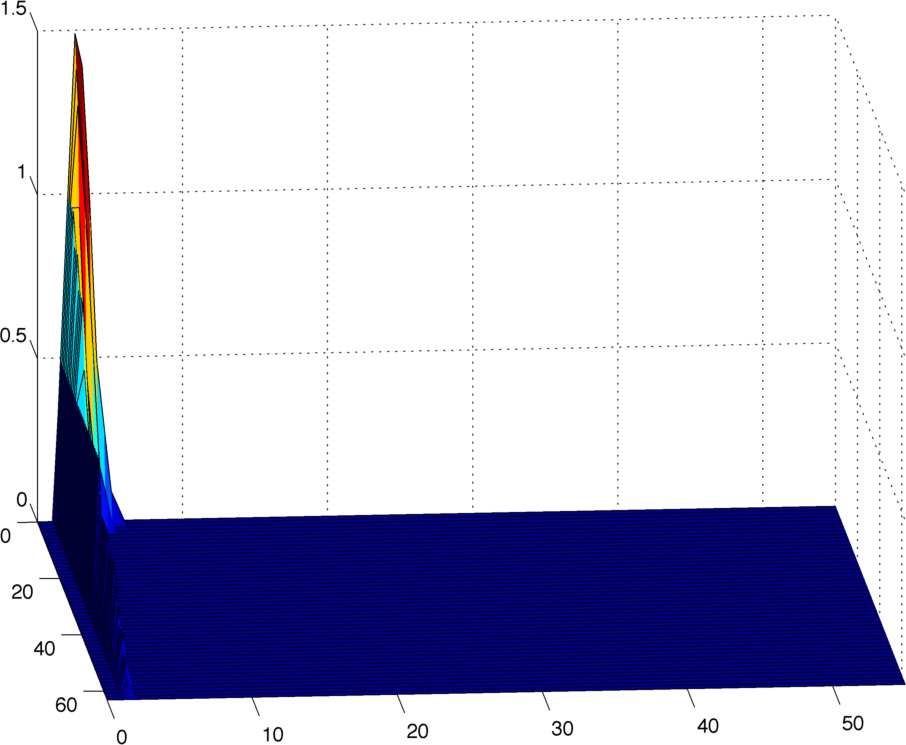}
\includegraphics[width=5cm]{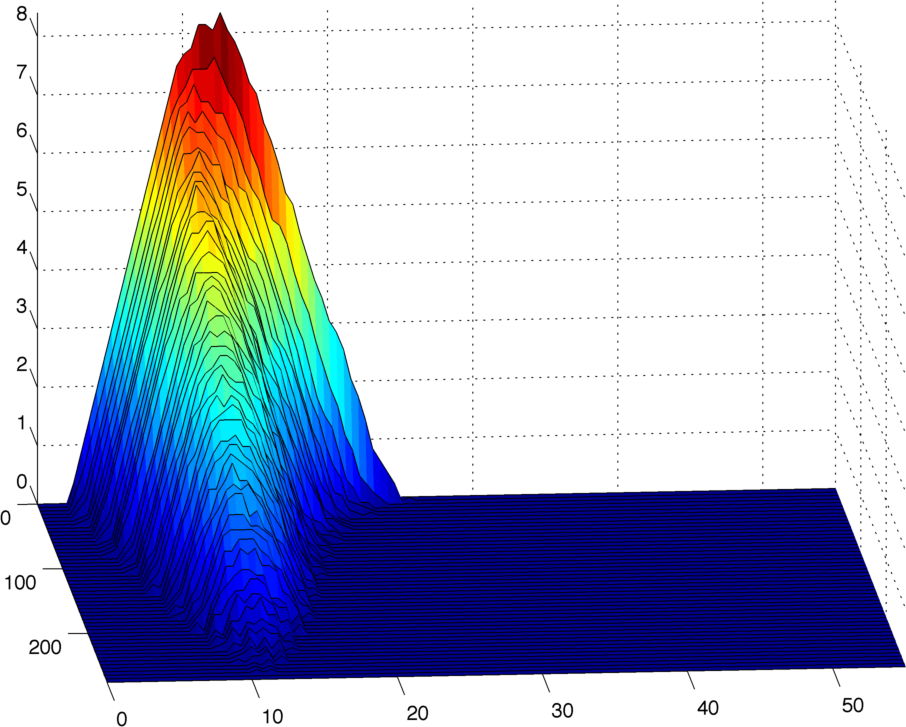}

\includegraphics[width=5cm]{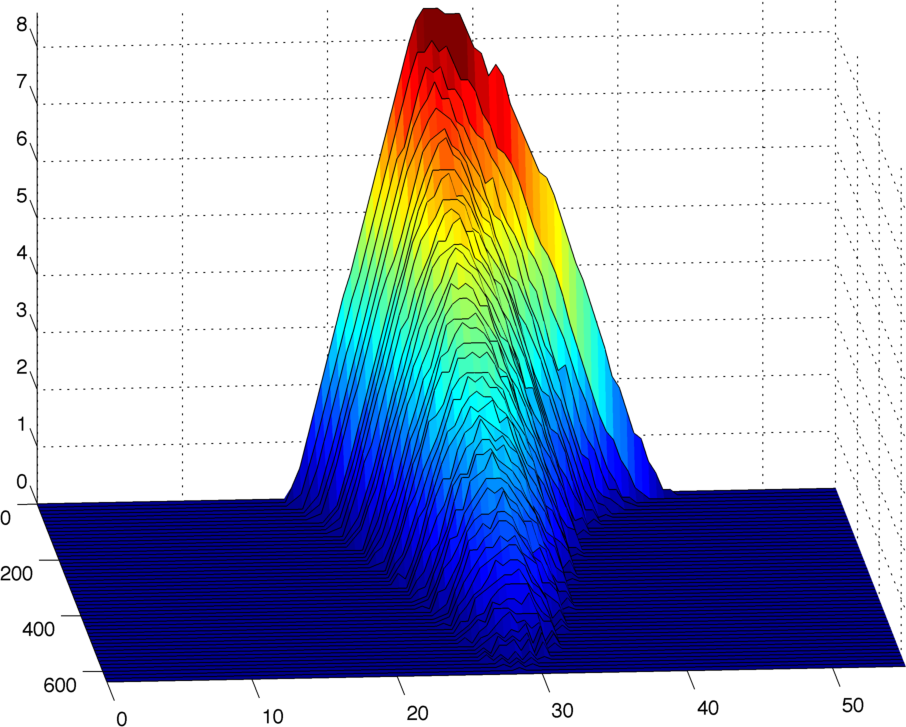}
\includegraphics[width=5cm]{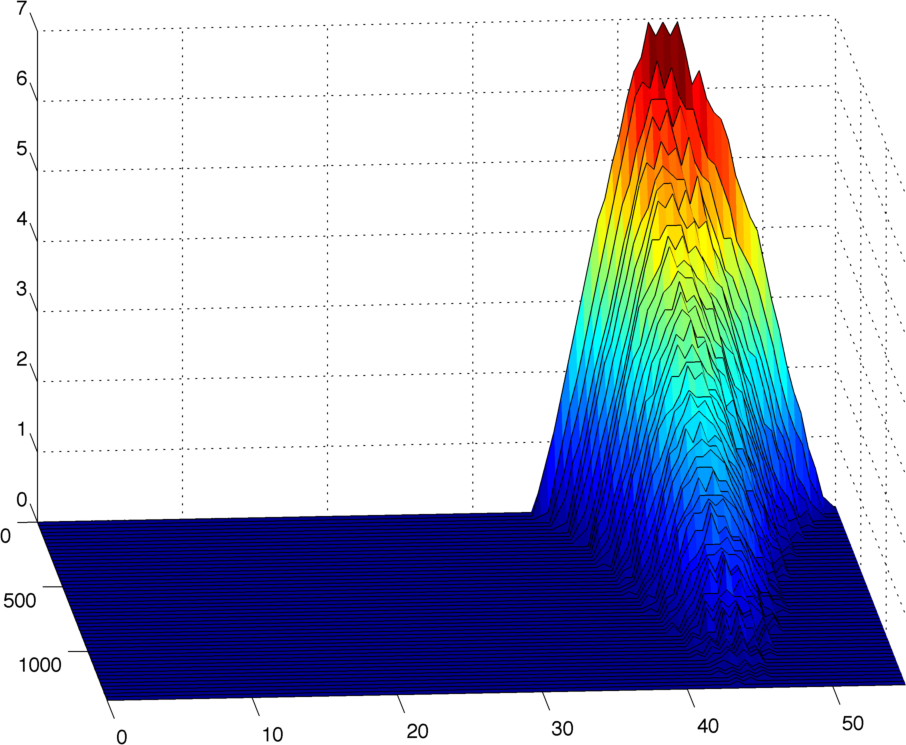}

\caption{The mean persistence landscapes in degrees 0--3 from 10
  copies of the random clique complex $G(n)$. Note that we have
  rescaled the filtration by a factor of 100.}
\label{fig:clique}
\end{figure}

\subsection{Gaussian Random Fields}
\label{sec:grf}

The topology of Gaussian random fields is of interest in statistics. 
The Euler characteristic of superlevel sets of a Gaussian random
field may be calculated using the Gaussian Kinematic
Formula of \citet{adlerTaylor:book}.
The persistent homology of Gaussian random fields has been
considered by \citet{abbsw:phrfc} and its expected Euler characteristic
has been obtained by \citet{bobrowskiBorman}.

Here we consider a stationary Gaussian random field on $[0,1]^2$ with
autocovariance function $\gamma(x,y) = e^{-400(x^2+y^2)}$. See Figure~\ref{fig:grf}.
We sample this field on a 100 by 100 grid, and calculate the
persistence landscape of the sublevel set. 
For homology in degree 0, we truncate the infinite interval at the
maximum value of the field.
We calculate the mean persistence landscapes in degrees 0 and 1
from 100 samples (see Figure~\ref{fig:grf}, where we have rescaled the
filtration by a factor of 100). 

% Since we have sampled the field on a finite grid, there is a bound on
% the values of $k$ for which $\lambda_k(t)$ is nonzero. However since
% the field is unbounded, so is the support of $\lambda$.
% So the situation is more delicate than that in
% Sections \ref{sec:points} and \ref{sec:clique}.

% One way to resolve this difficulty is to consider $\lambda_B$ and
% $D_B$ for some $B\gg 0$ as defined in Section~\ref{sec:functional}.
% We let $Y = \norm{\lambda_B(X)}_1 = \frac{1}{4}\pers_2(D_B(X))$. 
% The approximate 95\% confidence intervals of $E(Y)$ in degrees
% 0 and 1 are estimated to be $[33.87,35.37]$ and $[14.54,15.39]$.

In the Gaussian random field literature, it is more common to consider
superlevel sets. However, by symmetry, the expected persistence
landscape in this case is the same except for a change in the sign of the
filtration. 
%Furthermore $E(Y)$ is the same.

We repeat this calculation for a similar Gaussian random field on
$[0,1]^3$, this time using reduced homology. See
Figure~\ref{fig:grf}. This time we sample on a $25 \times 25 \times
25$ grid.  
% For $Y = \norm{\lambda_B(X)}_1 = \frac{1}{4}\pers_2(D_B(X))$, the
% approximate 95\% confidence intervals of $E(Y)$ in degrees 0, 1 and 2
% are estimated to be $[110.72,112.99]$, $[115.97,117.65]$ and
% $[48.98,50.27]$.

\begin{figure}
  \centering
  
  \includegraphics[width=7cm]{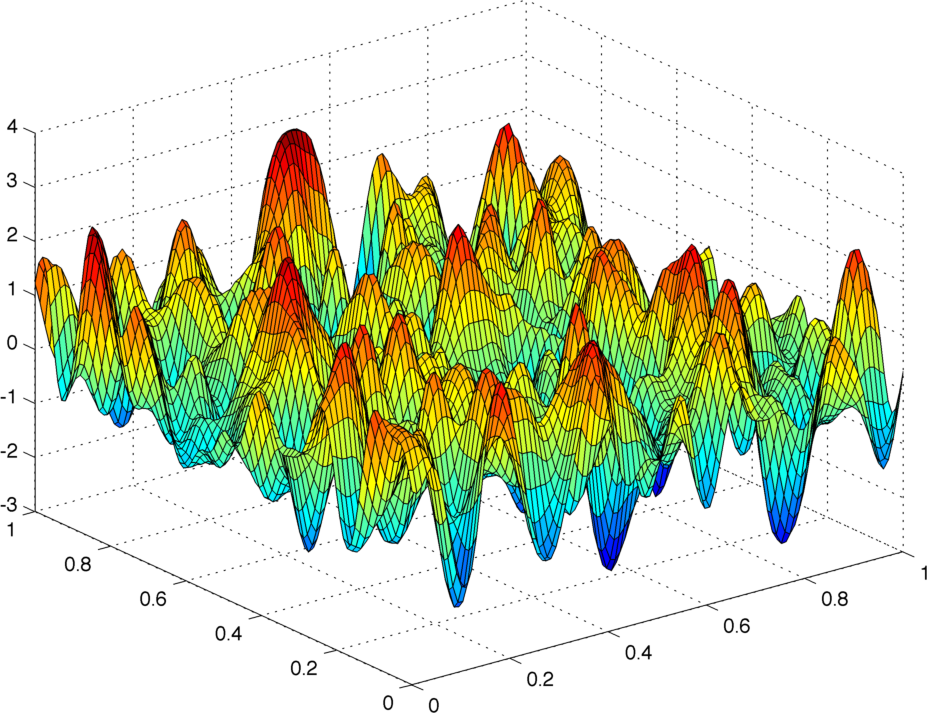}
  \includegraphics[width=5cm]{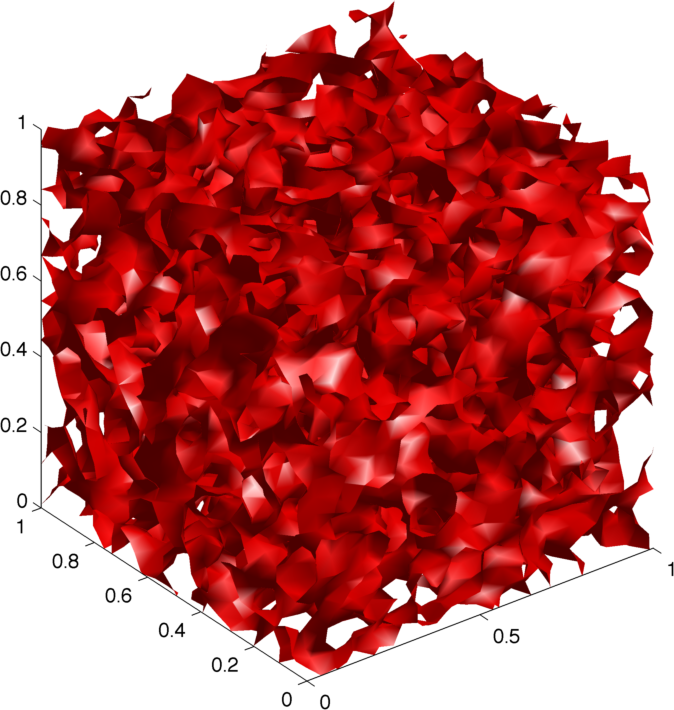}

  \includegraphics[width=4cm]{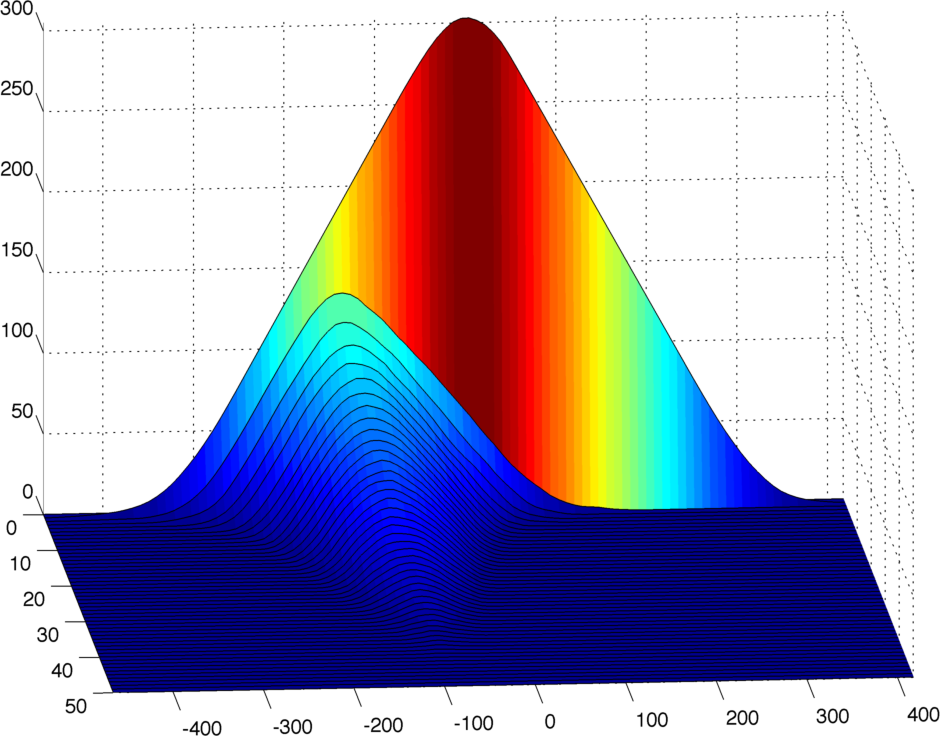}
  \includegraphics[width=4cm]{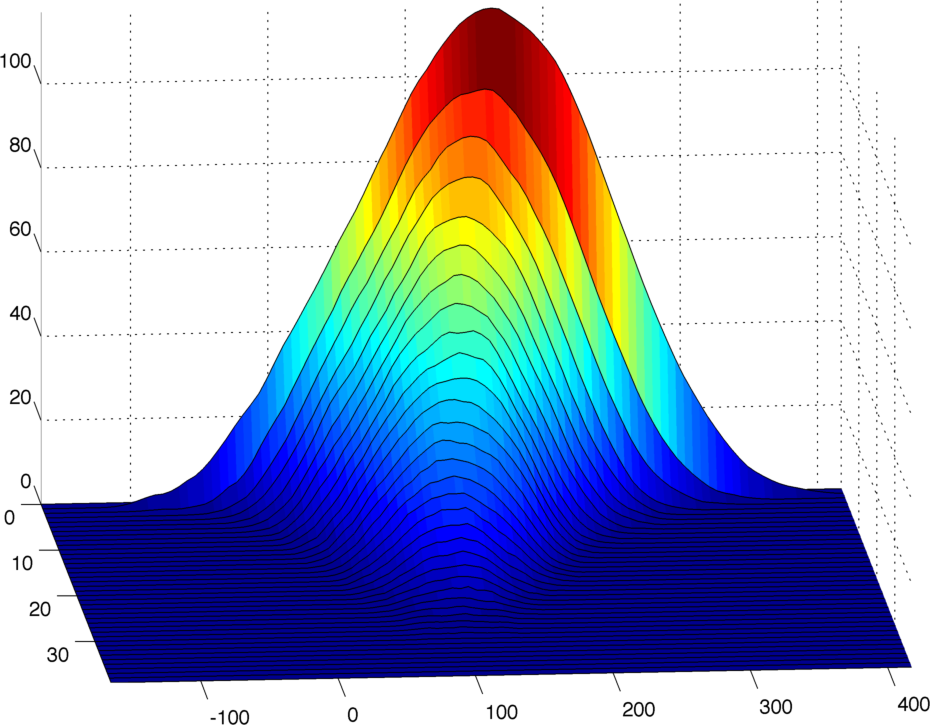}

  \includegraphics[width=4cm]{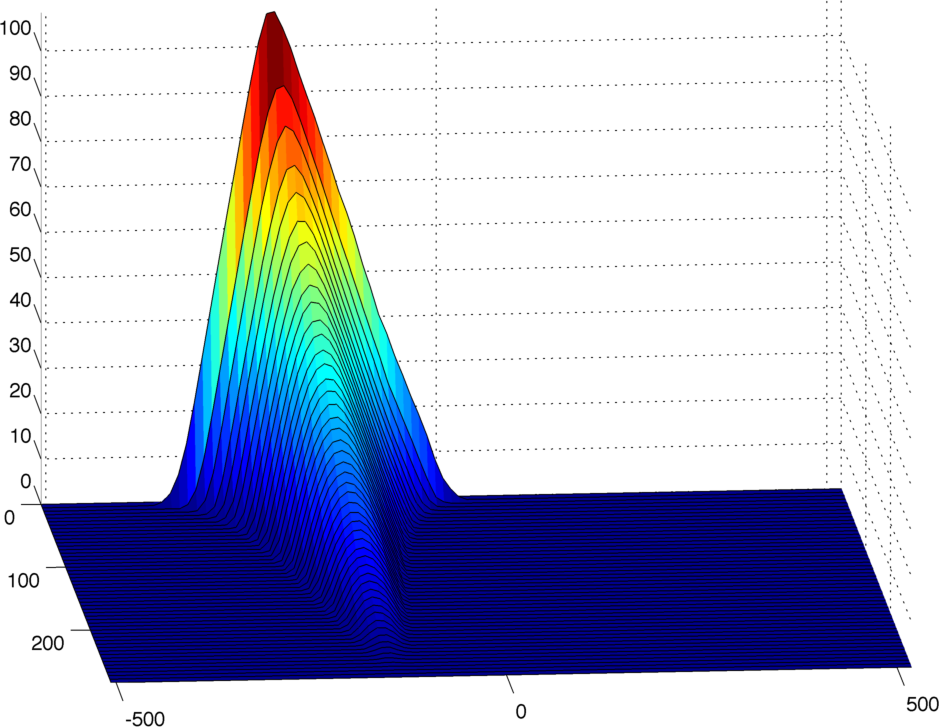}
  \includegraphics[width=4cm]{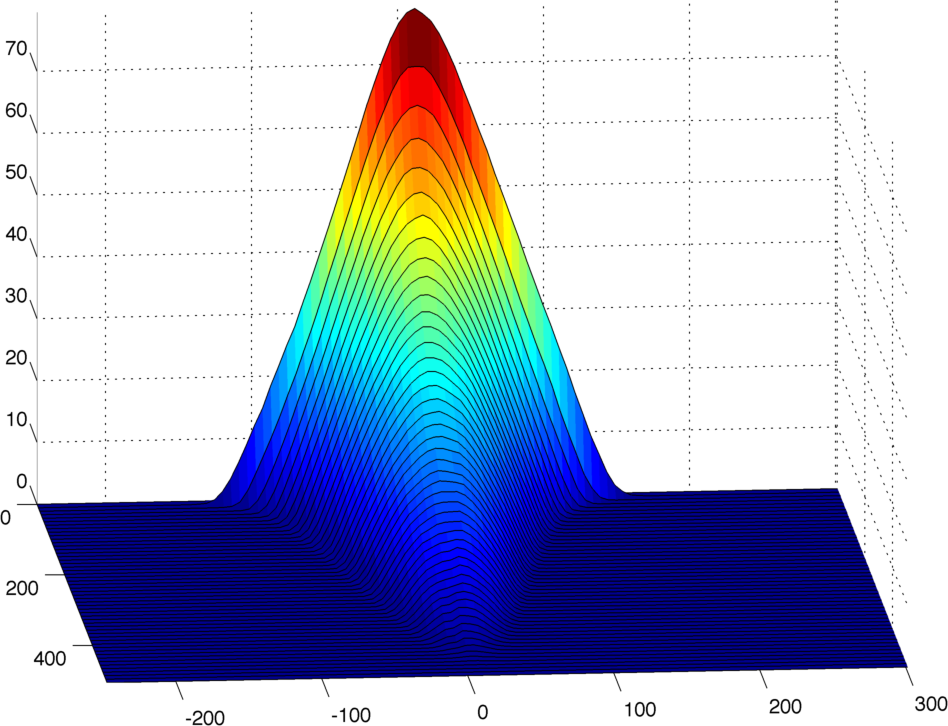}
  \includegraphics[width=4cm]{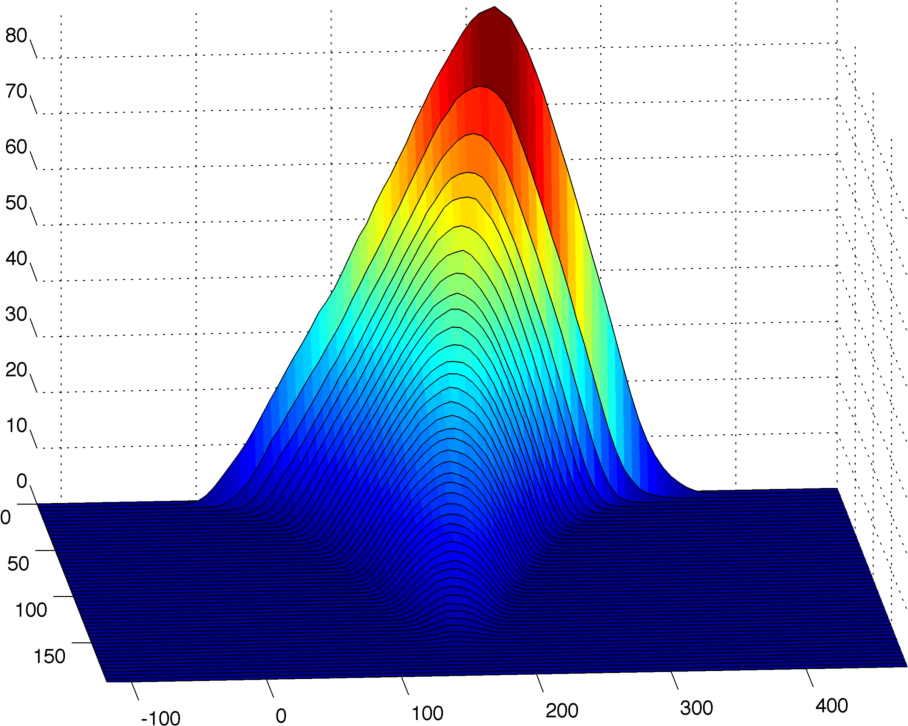}
  \caption{Mean landscapes of Gaussian random fields. The graph of a
    Gaussian random field on $[0,1]^2$ (top left) and its
    corresponding mean landscapes (middle row) in degrees
    0 and 1. The 0-isosurface of a Gaussian random field on $[0,1]^3$
    (top right) and the corresponding mean landscapes in
    degrees 0, 1 and 2 (bottom row).  }
  \label{fig:grf}
\end{figure}

\subsection{Torus and Sphere}
\label{sec:torus-sphere}

Here we combine persistence landscapes and statistical inference to
discriminate between iid samples of 1000 points from a torus and a
sphere in $\R^3$ with the same surface area, using the uniform surface
area measure as described by \citet{dhs:sampling} 
(see Figure~\ref{fig:torus-and-sphere}).  To be precise, we use the torus
given by $(r-2)^2+z^2=1$ in cylindrical coordinates, and the sphere
given by $r^2=2\pi$ in spherical coordinates.

\begin{figure} 
\begin{center}
\includegraphics[width=4cm]{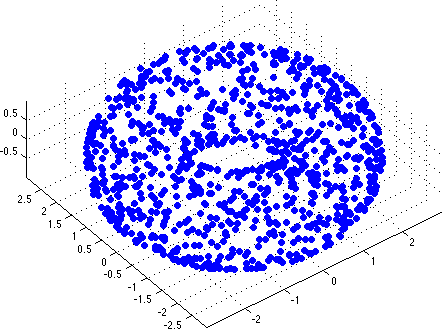}
\includegraphics[width=4cm]{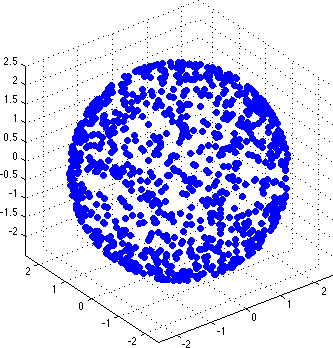}
\end{center}

\begin{center}
  \includegraphics[width=35mm]{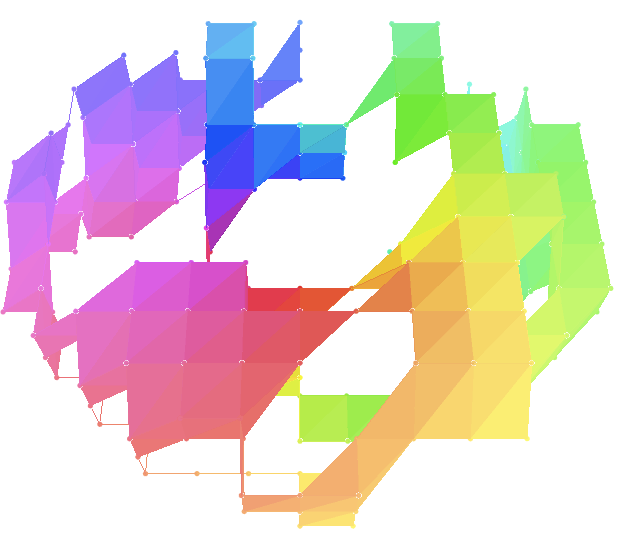}
  \includegraphics[width=35mm]{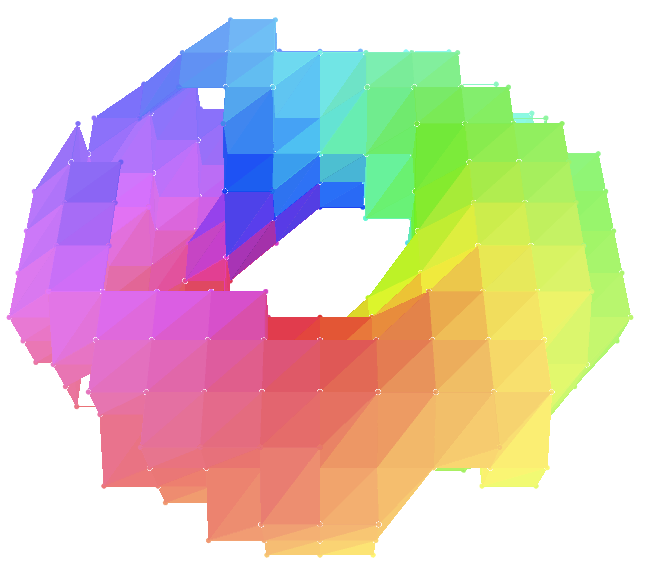}
  \includegraphics[width=35mm]{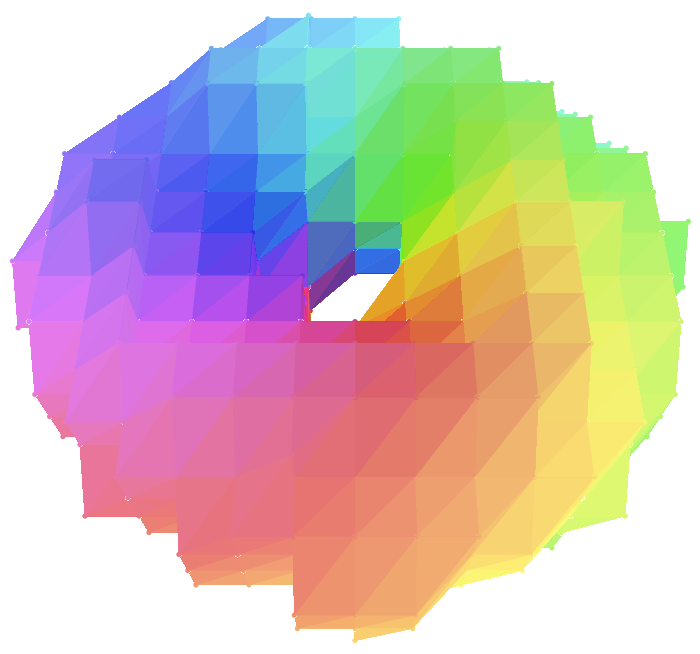}
\end{center}
\begin{center}
\includegraphics[width=4cm]{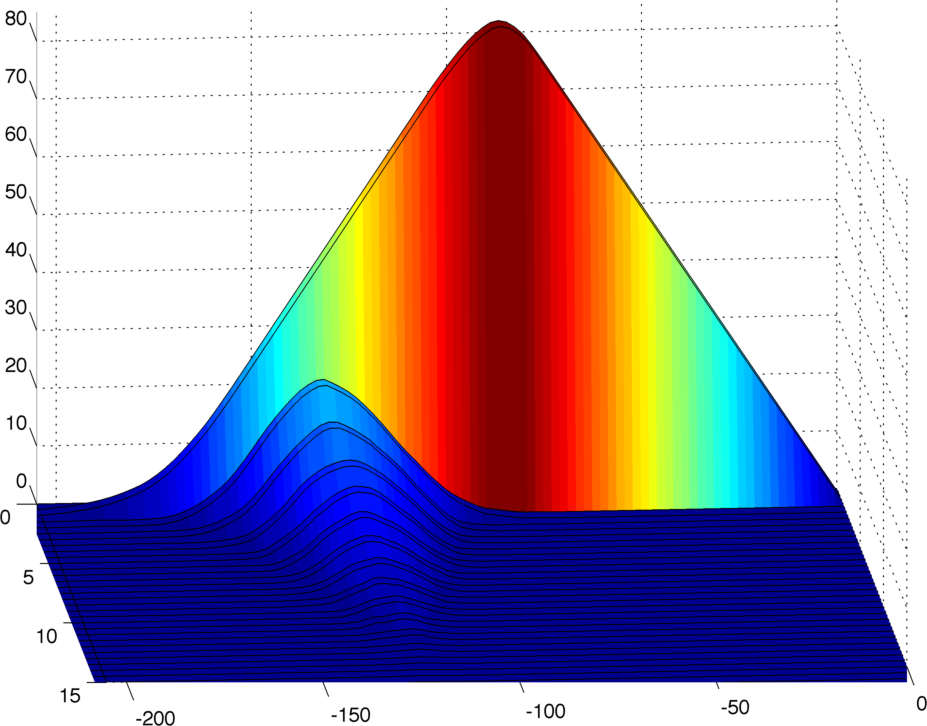}
\includegraphics[width=4cm]{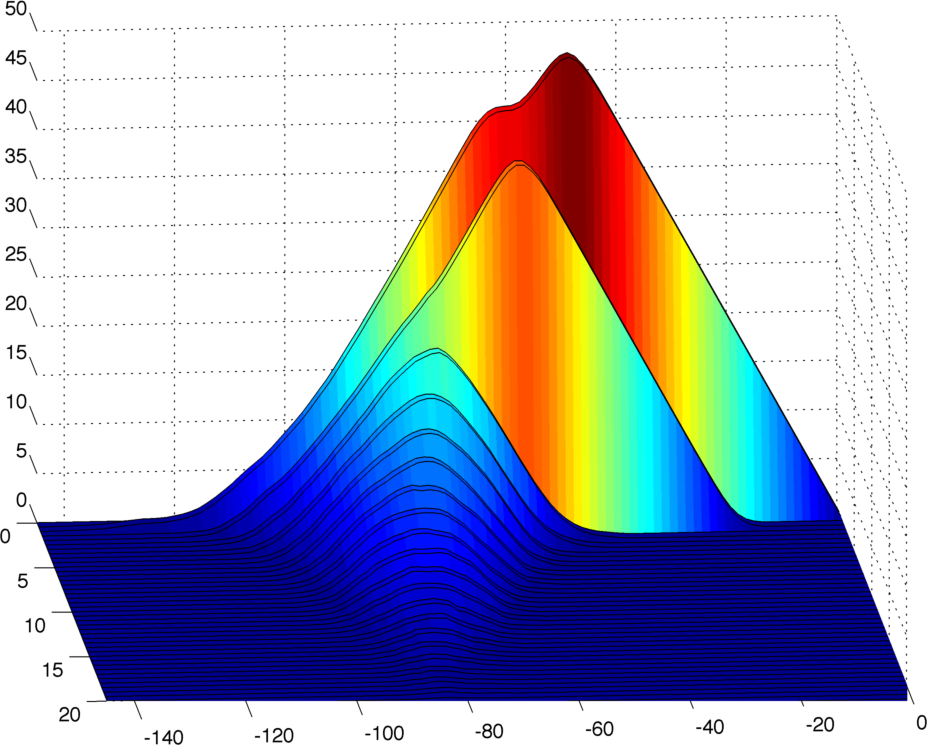}
\includegraphics[width=4cm]{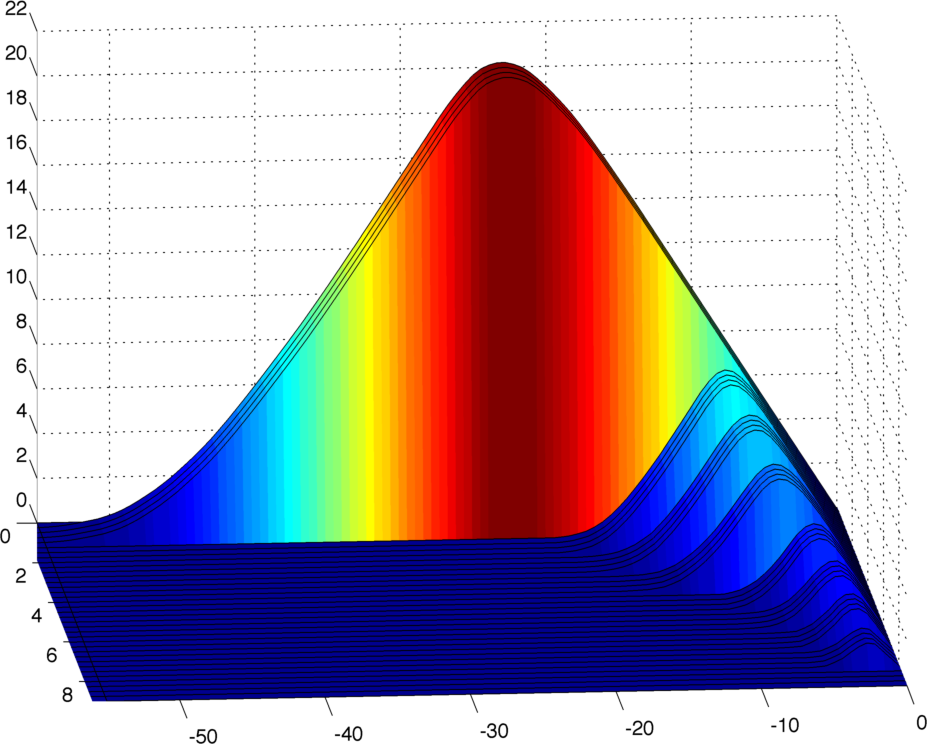}

\includegraphics[width=4cm]{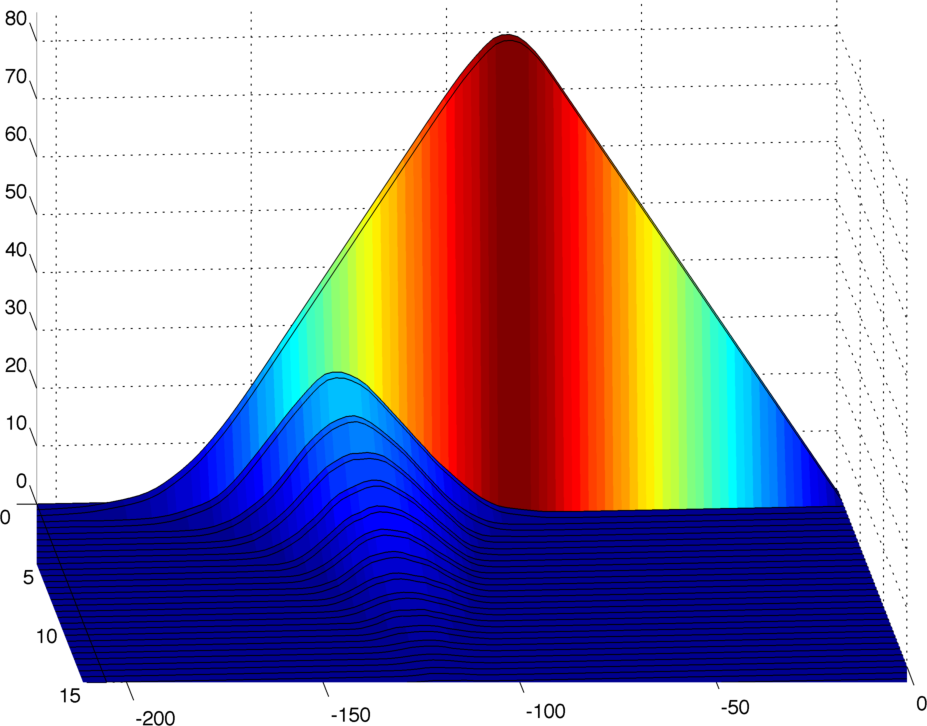}
\includegraphics[width=4cm]{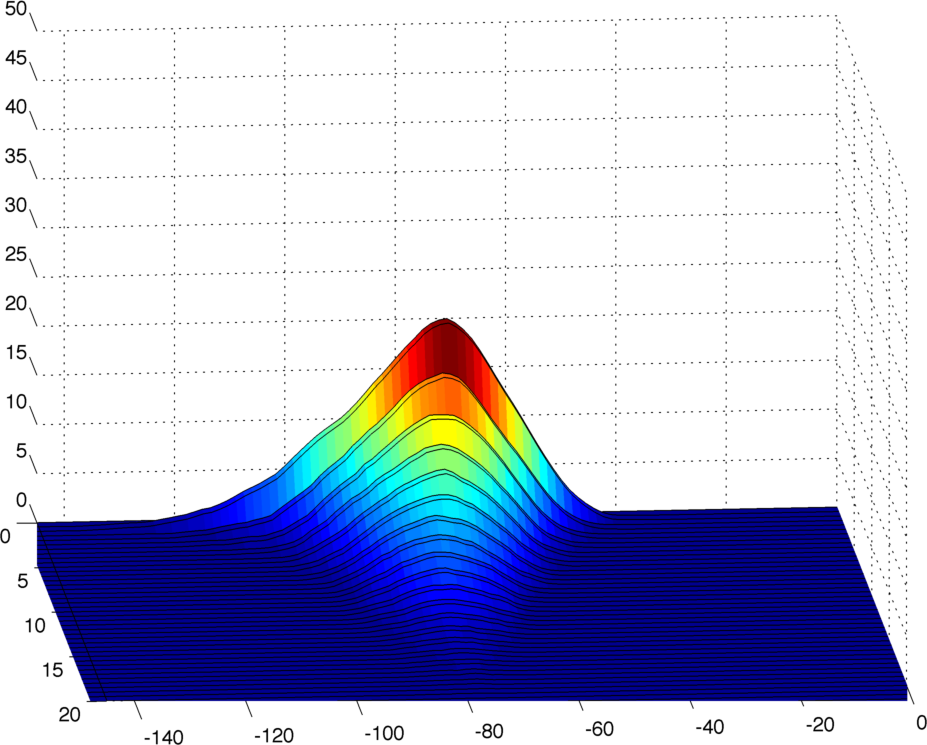}
\includegraphics[width=4cm]{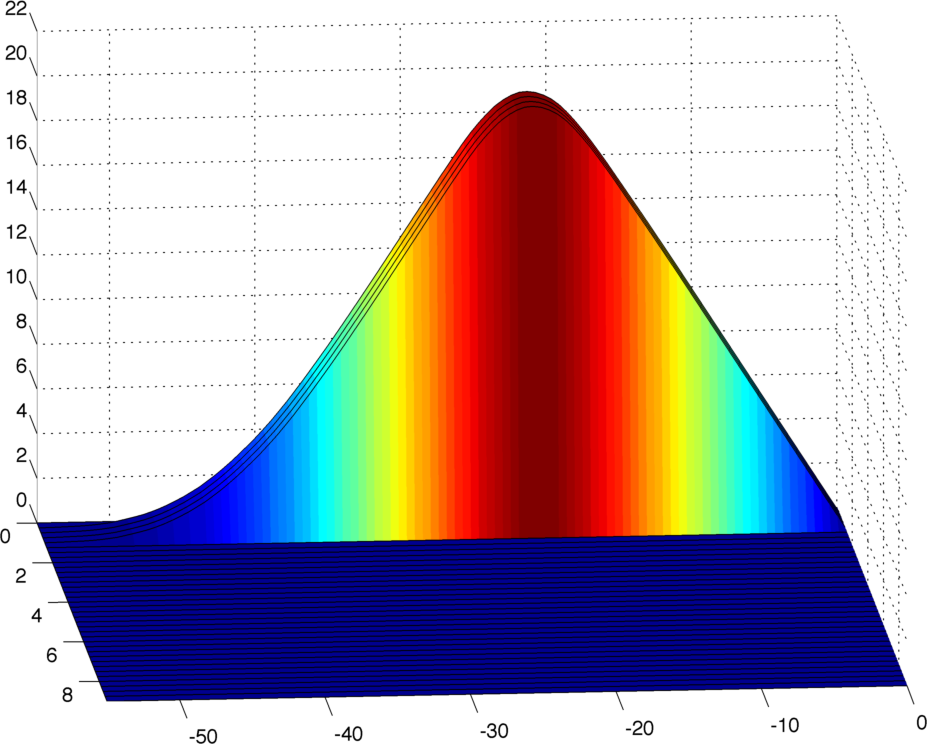}
\end{center}
\caption{We sample 1000 points for a torus and sphere, 100 times each,
  construct the corresponding filtered simplicial complexes and
  calculate persistent homology.  In columns 1, 2 and 3, we have the
  mean persistence landscape in dimension 0, 1 and 2 of the torus in
  row 3 and the sphere in row 4.  }
\label{fig:torus-and-sphere}
\end{figure}

For these points, we construct a filtered simplicial complex as follows. 
First we triangulate the underlying space using the
Coxeter--Freudenthal--Kuhn triangulation, starting with a cubical grid
with sides of length $\frac{1}{2}$.
Next we smooth our data using a triangular kernel with bandwidth 0.9.
We evaluate this kernel density estimator at the vertices of our simplicial complex.
Finally, we filter our simplicial complex as follows.
For filtration level $-r$, we include a simplex in our
triangulation if and only if the kernel density estimator has values
greater than or equal to $r$ at all of its vertices.
Three stages in the filtration for one of the samples are shown in 
(see Figure~\ref{fig:torus-and-sphere}).
We then calculate the persistence landscape of this filtered
simplicial complex for 100 samples and plot the mean landscapes
(see Figure~\ref{fig:torus-and-sphere}). We observe that the large peaks
correspond to the Betti numbers of the torus and sphere.
    
% In dimension 0, the $L^2$ distance between the mean landscapes of
% the torus and sphere is 48.8 and the square root of the Fr\'echet
% variances of the torus and sphere and 123.4 and 143.8, respectively.
% In dimension 1, we obtain a distance of 370.0 and square root of
% Fr\'echet variances of 69.6 and 57.6.
% In dimension 2, these numbers are 26.6, 23.7, and 24.8.

Since the support of the persistence landscapes is bounded, we can use
the integral of the landscapes to obtain a real valued random
variable that satisfies \eqref{eq:clt}. 
We use a two-sample z-test to test the null hypothesis
that these random variables have equal mean. 
For the landscapes in dimensions 0 and 2 we cannot reject the null
hypothesis.
%with p-values of 0.83 and 0.30.
In dimension 1 we do reject the null hypothesis with a p-value of $3
\times 10^{-6}$.

We can also choose a functional that only integrates the persistence
landscape $\lambda(k,t)$ for certain ranges of $k$.
In dimension 1, with $k=1$ or $k=2$ there is a statistically
significant difference (p-values of $10^{-8}$ and $3 \times 10^{-6}$),
but not for $k>2$.
In dimension 2, there is not a significant difference for $k=1$, but
there is a significant difference for $k>1$ (p-value $<10^{-4}$).

% Next we use the permutation test with 10,000 repetitions to determine if
% the $L^2$ distance between the mean landscapes of the
% two samples is statistically significant.
% %
% Comparing the two samples for dimension 0, the difference is not
% significant
% (p value 0.6106).
% %
% In dimension 1, the difference is significant with a p value of 0.0000.
% Restricting to each of $k=1$ and $k=2$ we have p values of 0.0000 and
% 0.0217, respectively.
% Restricting to $k>2$, there is no significant difference
% (p value 0.6570).
% %
% In dimension 2, we get a p value of 0.0088. 
% However if we restrict to $k=1$ 
% the difference is not significant
% (p value 0.9997).

% Noisy:

Now we increase the difficulty by adding a fair amount of Gaussian
noise to the point samples
(see Figure~\ref{fig:noisy}) and using only 10 samples for each surface.
\begin{figure}
\begin{center}
\includegraphics[width=4cm]{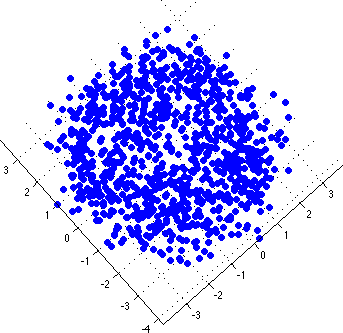}
\includegraphics[width=4cm]{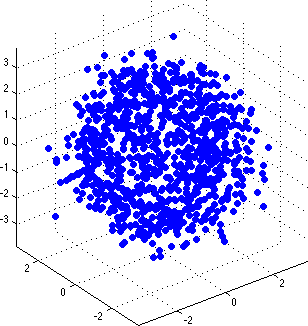}
  \includegraphics[width=4cm]{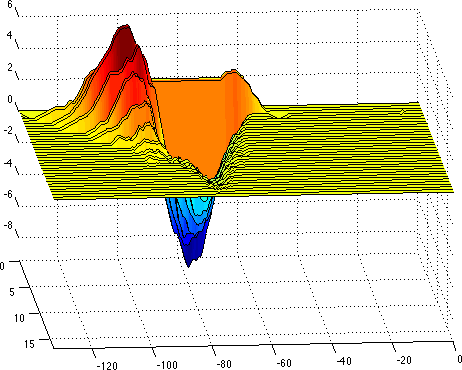}

\includegraphics[width=4cm]{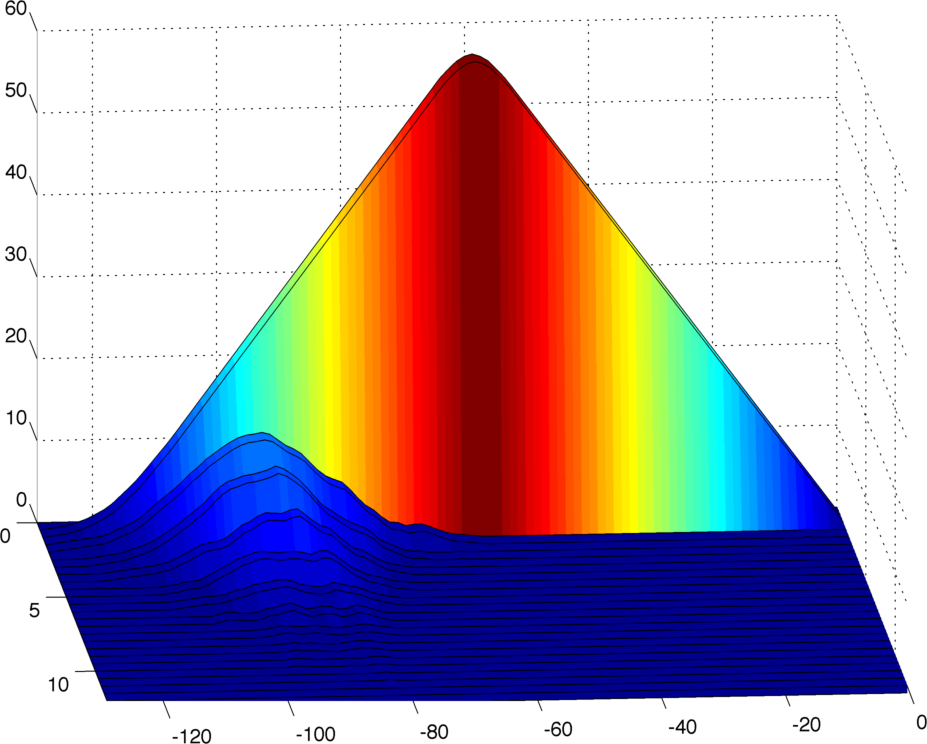}
\includegraphics[width=4cm]{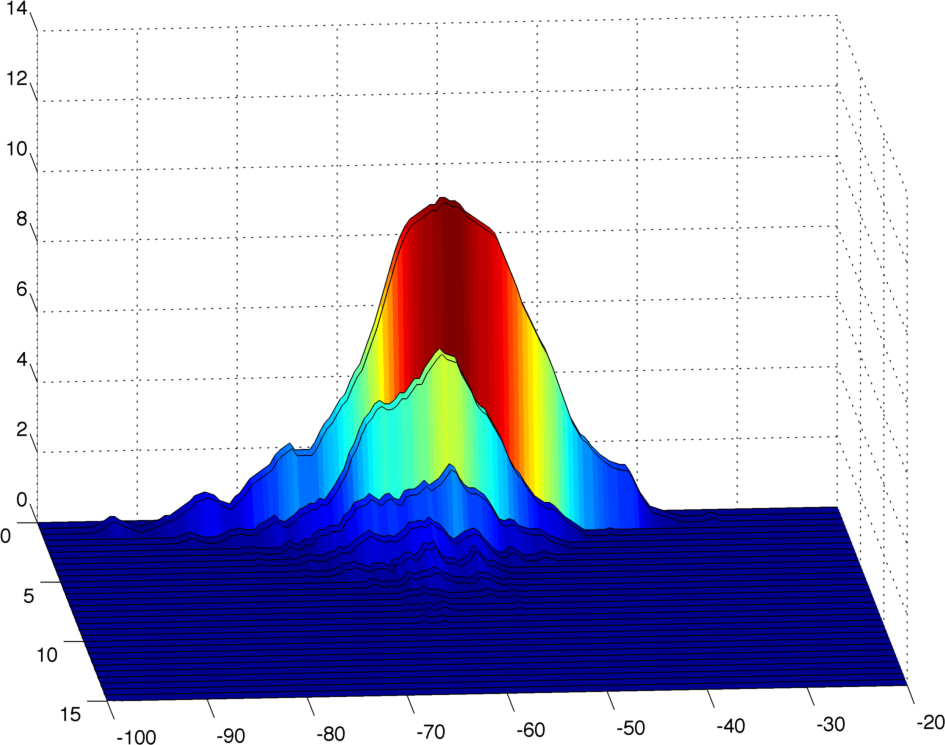}
\includegraphics[width=4cm]{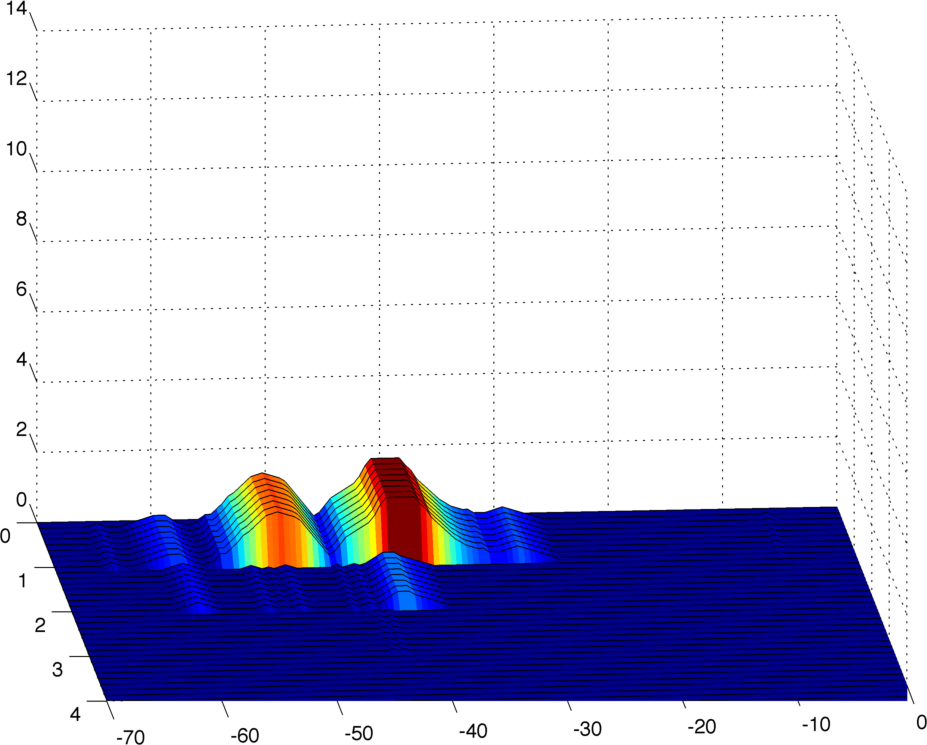}

\includegraphics[width=4cm]{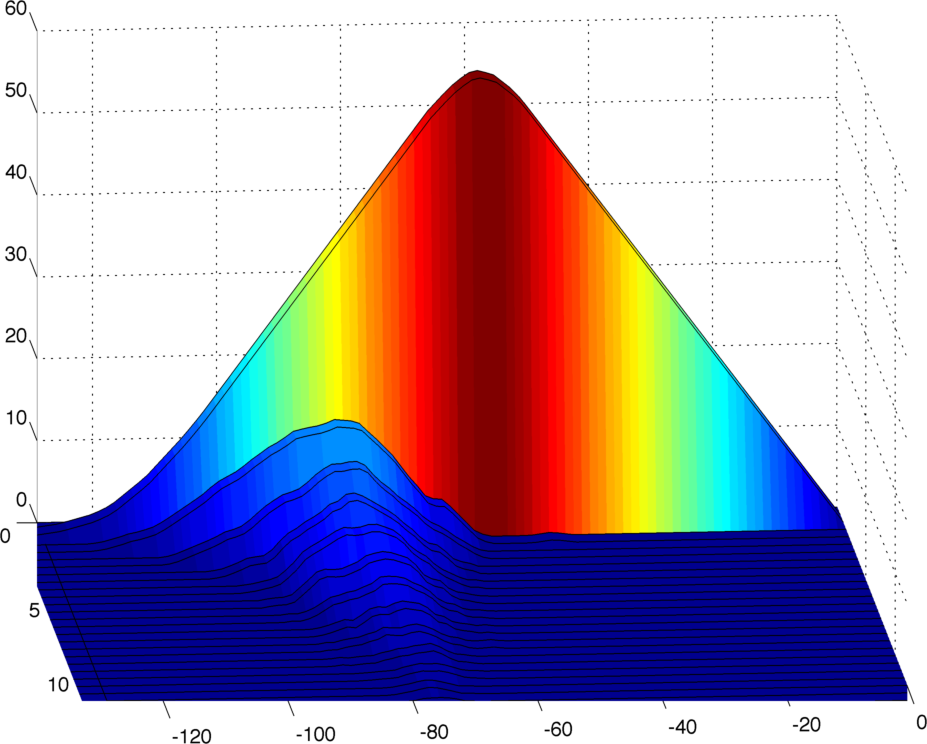}
\includegraphics[width=4cm]{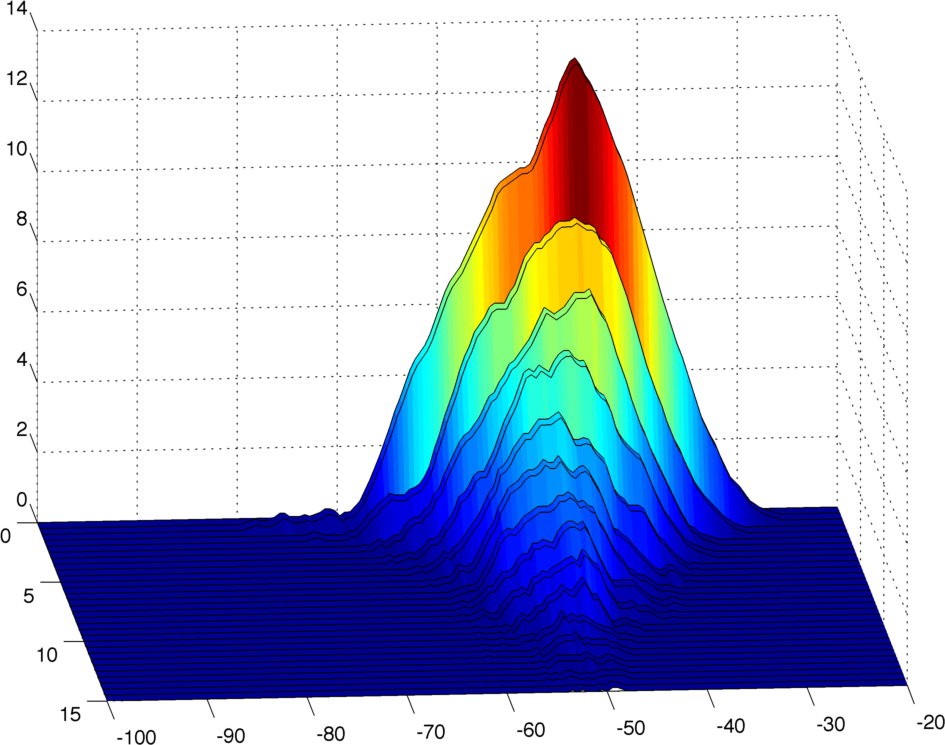}
\includegraphics[width=4cm]{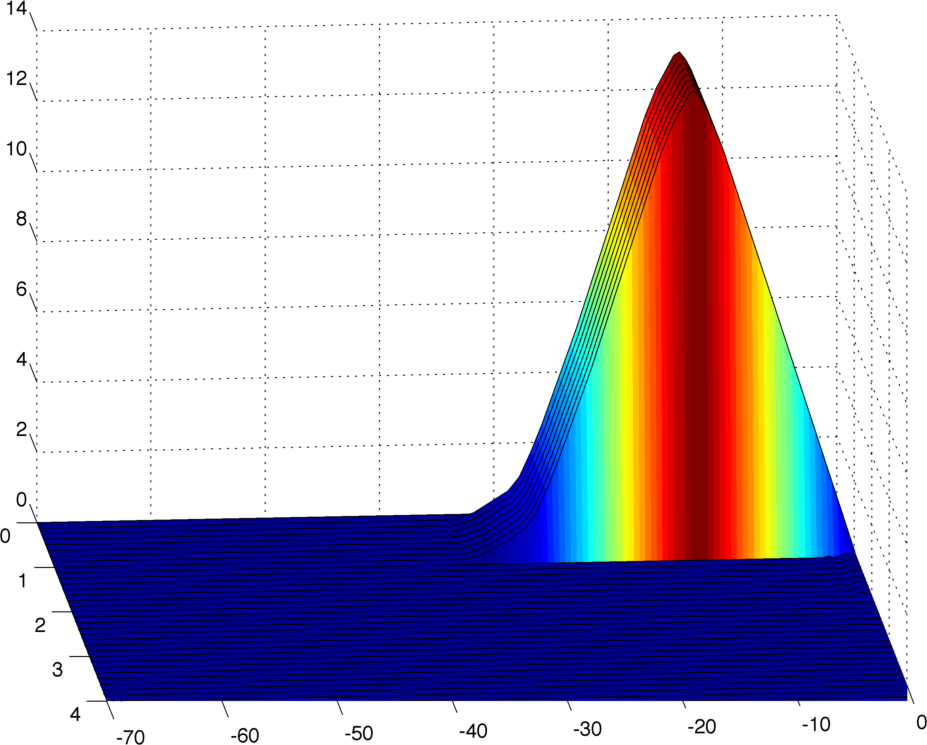}
\end{center}
\caption{We again sample 1000 points sampled from a torus (top left)
  and sphere (top middle),
  this time with Gaussian noise. We show the torus from the
perspective that makes it easiest to see the hole in the middle.
We calculate persistent homology from 10 samples.
%Here are the mean persistence landscapes.
In columns 1, 2 and 3, we have the mean persistence landscape in
dimension 0, 1 and 2, respectively, with the torus in row 2 and the
sphere in row 3.
The top right is a graph of the difference between the mean landscapes
in dimension 0.
}
\label{fig:noisy}
\end{figure}
This time we calculate the $L^2$ distances between the mean landscapes.
% and the square root of the Fr\'echet variances from the torus and the
% sphere.
% In dimension 0, they are 69.7, 67.3 and 78.1.
% In dimension 1, we have 59.5, 23.5 and 26.4.
% In dimension 2, the values are 44.5, 7.2 and 10.2.
%
We use the permutation test with 10,000 repetitions to determine if
this distance is statistically significant.
There is a significant difference in dimension 0, with a
p value of 0.0111.
This is surprising, since the mean landscapes look very
similar. However, on closer inspection, they are shifted
slightly (see Figure~\ref{fig:noisy}). 
Note that we are detecting a geometric difference, not a topological
one.
This shows that this statistic is quite powerful.
There is also a significant difference in
dimensions 1 and 2, with p values of 0.0000 and 0.0000, respectively.

\section{Landscape Distance and Stability}
\label{sec:stability}

In this section we define the landscape distance and use it to show that
the persistence landscape is a stable summary statistic. We also show
that the landscape distance gives lower bounds for the bottleneck and
Wasserstein distances. We defer the proofs of 
% Theorems \ref{thm:inf-stability}, \ref{thm:inf-bound} and
% \ref{thm:stability} and Corollary~\ref{cor:Wp}
the results of this section
to the appendix.

Let $M$ and $M'$ be persistence modules as defined in
Section~\ref{sec:pm} and let $\lambda$ and $\lambda'$ be their
corresponding persistence landscapes as defined in Section~\ref{sec:landscape}.
For $1 \leq p \leq \infty$,
define the \emph{$p$-landscape distance} between $M$ and $M'$ by
\begin{equation*}
  \Lambda_p(M,M') = \norm{\lambda-\lambda'}_p.
\end{equation*}
Similarly, if $\lambda$ and $\lambda'$ are the persistence landscapes
corresponding to persistence diagrams $D$ and $D'$ (Section~\ref{sec:barcode}), then we define
\begin{equation*}
  \Lambda_p(D,D') = \norm{\lambda-\lambda}_p.
\end{equation*}

Given a real valued function $f:X \to \R$ on a topological space
$X$, let $M(f)$ denote be the corresponding persistence module
defined at the end of Section~\ref{sec:pm}.
\begin{theorem}[$\infty$-Landscape Stability Theorem] 
\label{thm:inf-stability}
  Let $f,g:X \to \R$. Then
\begin{equation*}
  \Lambda_{\infty}(M(f),M(g)) \leq \norm{f-g}_{\infty}.
\end{equation*}
\end{theorem}
Thus the persistence landscape is stable with respect to the supremum norm.
We remark that there are no assumptions on $f$ and $g$, not even the
q-tame condition of \citet{csgo:persistenceModules}.

Let $D$ be a persistence diagram.
For $x = (b,d) \in D$, let $\ell = d-b$ denote the \emph{persistence} of $x$.
If $D = \{x_j\}$, let $\Pers_k(D) = \sum_j \ell_j^k$ denote the
\emph{degree-$k$ total persistence} of $D$.

Now let us consider a persistence diagram to be an equivalence
class of multisets of pairs $(b,d)$ with $b \leq d$, 
where $D \sim D \amalg \{(t,t)\}$ for any $t \in \R$. 
That is, to any persistence diagram, we can freely adjoin points on the diagonal. This is reasonable, since points on the diagonal have zero persistence.
Each persistence diagram has a unique representative $\hat{D}$ without
any points on the diagonal.
We set $\abs{D} = \abs{\hat{D}}$.
We also remark that $\Pers_k(D)$ is well defined.

By allowing ourselves to add as many points on the
diagonal as necessary, there exists bijections between any two
persistence diagrams.
% As in \citet{cseh:stability} and
% \citet{csehm:lipschitz}, we assume that all of our persistence diagrams
% have a finite representative.
%
Any bijection $\varphi:D \isomto D'$ can be represented by $\varphi:x_j \mapsto
x'_j$, where $j \in J$ with $\abs{J} = \abs{D} + \abs{D'}$.
For a given $\varphi$, let $x_j = (b_j,d_j)$, $x'_j =
(b'_j,d'_j)$ and $\eps_j = \norm{x_j - x'_j}_{\infty} = \max(\abs{b_j-b'_j},\abs{d_j-d'_j})$.
%Define $W_p^{\varphi}(D,D') = (\sum_{j=1}^n\eps_j^p)^{\frac{1}{p}}$.

The \emph{bottleneck distance} \citep{cseh:stability} between
persistence diagrams $D$ and $D'$ is given by
\begin{equation*}
  W_{\infty}(D,D') = \inf_{\varphi:D \isomto D'} \sup_{j} \eps_j,
\end{equation*}
where the infimum is taken over all bijections from $D$ to $D'$.
It follows that for the empty persistence diagram
$\emptyset$, $W_{\infty}(D,\emptyset) = \frac{1}{2} \sup_{j} \ell_j$.

The $\infty$-landscape distance is bounded by the bottleneck distance.

\begin{theorem} \label{thm:inf-bound}
  For persistence diagrams $D$ and $D'$,
  \[
  \Lambda_{\infty}(D,D') \leq W_{\infty}(D,D').
  \]
\end{theorem}

For $p \geq 1$, the \emph{$p$-Wasserstein
  distance} \citep{csehm:lipschitz} between $D$ and $D'$ is given by
\begin{equation*}
  W_p(D,D') = \inf_{\varphi:D \isomto D'} \left[ \sum_j
    \eps_j^p \right]^{\frac{1}{p}}.
\end{equation*}

We remark that
% if $D=\{x_j\}$ and $D'=\{x_j+\eps_j\}$, then 
the Wasserstein distance gives equal weighting to the $\eps_j$ while
the landscape distance gives a stronger weighting to $\eps_j$ if $x_j$
has larger persistence.
The landscape distance is most closely related to a weighted version
of the Wasserstein distance that we now define.
The \emph{persistence weighted $p$-Wasserstein distance} between $D$
and $D'$ is given by
\begin{equation*}
  \overline{W}_p(D,D') = \inf_{\varphi:D \isomto D'} \left[ \sum_j \ell_j
    \eps_j^p \right]^{\frac{1}{p}}.
\end{equation*}
Note that it is asymmetric.

For the remainder of the section we assume that $D$ and $D'$ are
finite.  The following result bounds the $p$-landscape
distance. Recall that $\ell_j$ is the persistence of $x_j \in D$ and
when $\varphi: x_j \mapsto x'_j$, $\eps_j = \norm{x_j-x'_j}_{\infty}$

%Our main result in the following, whose proof is in the appendix.

\begin{theorem} \label{thm:stability}
  If $n = \abs{D} + \abs{D}$ then
  \begin{equation*}
    \Lambda_p(D,D')^p \leq \min_{\varphi:D \isomto D'} \left[
    \sum_{j=1}^n \ell_j \eps_j^p + \frac{2}{p+1} \sum_{j=1}^n
    \eps_j^{p+1} \right].
\end{equation*}
\end{theorem}

From this we can obtain a lower bound on the $p$-Wasserstein distance.

%From Theorem~\ref{thm:stability}, 
% If $W_p(D,D') \leq 1$ then for the
% minimizer for $W_p(D,D')$, $\eps_j \leq 1$ for all $j$. So we get the
% following corollary.

\begin{corollary} \label{cor:Wp}
  % If $W_p(D,D') \leq 1$ then $\Lambda_p(D,D')^p \leq
  % 2(W_{\infty}(D,\emptyset)+\frac{1}{p+1})W_p(D,D')^p$. Thus
  $W_p(D,D')^p \geq
  \min \left(1,\frac{1}{2}\left[W_{\infty}(D,\emptyset)+\frac{1}{p+1}\right]^{-1}
  \Lambda_p(D,D')^p\right)$.
\end{corollary}

% From Theorem~\ref{thm:stability}, using the terminology
% from \citet{csehm:lipschitz}
% and following the proof of the Wasserstein Stability Theorem therein, 
% we have the following.

For our final stability theorem, we use ideas from
\citet{csehm:lipschitz}.  
Let $f:X \to \R$ be a function on a topological space.
We say that $f$ is \emph{tame} if for all but finitely many $a \in \R$,
the associated persistence module $M(f)$ is constant and finite
dimensional on some open interval containing $a$.
For such an $f$, let $D(f)$ denote the corresponding persistence
diagram.  
If $X$ is a metric space we say that $f$ is \emph{Lipschitz} if there is some
constant $c$ such that $\abs{f(x)-f(y)} \leq c\,  d(x,y)$ for all $x,y\in
X$.
We let $\Lip(f)$ denote the infimum of all such $c$.
We say that a metric space $X$ \emph{implies bounded degree-$k$ total
  persistence} if there is a constant $C_{X,k}$ such that $\Pers_k(D(f))
\leq C_{X,k}$ for all tame Lipschitz functions $f:X \to \R$ such that
$\Lip(f) \leq 1$.
For example, as observed by \citet{csehm:lipschitz}, if $X$ is the $n$-dimensional sphere, then $X=S^n$ has bounded $k$-persistence for $k=n+\delta$ for any $\delta>0$, but does not have bounded $k$-persistence for $k<n$.

\begin{theorem}[$p$-Landscape stability theorem] \label{thm:landscape-stability}
  Let $X$ be a triangulable, compact metric space that implies bounded
  degree-$k$ total persistence for some real number $k \geq 1$, and
  let $f$ and $g$ be two tame Lipschitz functions. Then
  \begin{equation*}
    \Lambda_p(D(f),D(g))^p \leq C \norm{f-g}_{\infty}^{p-k},
  \end{equation*}
  for all $p \geq k$, where 
  % $C = C_X \max\{\Lip(f)^k,\Lip(g)^k,\Lip(f)^{k+1},\Lip(g)^{k+1}\}
  % (W_{\infty}(D,\emptyset) + \frac{1}{p+1})$.
  $C = C_{X,k} \norm{f}_{\infty} (\Lip(f)^k+\Lip(g)^k) + C_{X,k+1}
  \frac{1}{p+1} (\Lip(f)^{k+1}+\Lip(g)^{k+1})$.
\end{theorem}

Thus the persistence diagram is stable with respect to the
$p$-landscape distance if $p>k$, where $X$ has bounded degree-$k$
total persistence.  This is the same condition as for the stability of
the $p$-Wasserstein distance in \citet{csehm:lipschitz}.
Equivalently, the persistence landscape is stable with respect to the
$p$-norm if $p>k$, where $X$ has bounded degree-$k$
total persistence.

% From Theorem~\ref{thm:stability} we also have the following corollary.

% \begin{corollary} 
%   $\Lambda_p(D,D')^p \leq \overline{W}_p(D,D')^p +
%   \frac{2}{p+1}W_{p+1}^{\varphi}(D,D')^{p+1}$,
%   where $\varphi:D \to D'$ is a bijection that minimizes the
%   persistence weighted $p$-Wasserstein distance.
% \end{corollary}

% Since
% $\overline{W}_p(D,D')^p \leq 2 \, W_{\infty}(D,\emptyset)
% W_p(D,D')^p$,
% we have the following.
% \begin{corollary}
%   $\Lambda_p(D,D')^p \leq 2 \; W_{\infty}(D,\emptyset) \; W_p(D,D')^p
%   + \frac{2}{p+1} \, W_{p+1}^{\varphi}(D,D')^{p+1}$.
% \end{corollary}

% Acknowledgements should go at the end, before appendices and references

\subsection*{Acknowledgments}

{The author would like to thank Robert Adler, Frederic Chazal,
  Herbert Edelsbrunner, Giseon Heo, Sayan Mukherjee and Stephen Rush
  for helpful discussions. Thanks to Junyong Park for suggesting
  Hotelling's $T^2$ test. Also thanks to the anonymous referees who
  made a number of helpful comments to improve the exposition. In addition,  the
  author gratefully acknowledges the support of the Air Force
  Office of Scientific Research (AFOSR grant FA9550-13-1-0115). }


%\newpage

\appendix

\section{Proofs}
\label{sec:proofs}

\begin{proof}[Proof of Lemma~\ref{lem:PLproperties}(\ref{it:lipschitz})]
We will prove that $\lambda_k$ is 1-Lipschitz.
That is, $\abs{\lambda_k(t) - \lambda_k(s)} \leq \abs{t-s}$, for all $s,t \in \R$.

Let $s,t \in \R$. Without loss of generality, assume that
$\lambda_k(t) \geq \lambda_k(s) \geq 0$.
If $\lambda_k(t) \leq \abs{t-s}$, then $\lambda_k(t) - \lambda_k(s) \leq \lambda_k(t) \leq \abs{t-s}$ and we are done.
So assume that $\lambda_k(t) > \abs{t-s}$.
 
% Let $0 < \eps < \lambda_k(t) - \abs{t-s}$.  By
% Lemma~\ref{lem:beta-decreasing}, our assumption, and
% Definition~\ref{def:landscape},
%   \begin{equation} \label{eq:beta-t}
%   \beta^{t-(\lambda_k(t)-\eps),t+(\lambda_k(t)-\eps)} \geq \beta^{t-\abs{t-s},t+\abs{t-s}} \geq k.
% \end{equation}
% Since $-\abs{t-s} \leq t-s \leq \abs{t-s}$, $s - \abs{t-s} \leq t \leq s + \abs{t-s}$.
% Also, $-\lambda_k(t) + \abs{t-s} + \eps < 0 < \lambda_k(t) - \abs{t-s} - \eps$.
% Using these inequalities, we have, 
% \begin{multline*}
%   t-\lambda_k(t)+\eps \leq s-\lambda_k(t)+\abs{t-s}+\eps < s < \\
% < s+\lambda_k(t)-\abs{t-s}-\eps \leq t+\lambda_k(t)-\eps.
% \end{multline*}
% Together with Lemma 1 %~\ref{lem:beta-decreasing} 
% and \eqref{eq:beta-t}, we see that
% \begin{equation*}
%   \beta^{s-(\lambda_k(t)-\abs{t-s}-\eps),
%     s+(\lambda_k(t)-\abs{t-s}-\eps)} \geq \beta^{t-(\lambda_k(t)-\eps),t+(\lambda_k(t)-\eps)} \geq k.
% \end{equation*}

Let $0 < h < \lambda_k(t)-\abs{t-s}$.
Then $t-\lambda_k(t) < s-h < s+h < t+\lambda_k(t)$.
Thus, by Lemma~\ref{lem:rank} and Definition~\ref{def:landscape}, 
$\beta^{s-h,s+h} \geq k$.
It follows that $\lambda_k(s) \geq \lambda_k(t) - \abs{t-s}$.
Thus $\lambda_k(t) - \lambda_k(s) \leq \abs{t-s}$.
\end{proof}

Theorems \ref{thm:inf-stability} and \ref{thm:inf-bound} follow from
the next result which is of independent interest. 
Following \citet{ccsggo:interleaving}, we say that two persistence
modules $M$ and $M'$ are \emph{$\eps$-interleaved} if for all $a \in
\R$ there exist linear maps $\varphi_a: M_a \to M'_{a+\eps}$ and $\psi:M'_a
\to M_{a + \eps}$ such that for all $a \in \R$, $\psi_{a+\eps}\circ
\varphi_a = M(a \leq a+2\eps)$ and $\varphi_{a+\eps} \circ \psi_a = M'(a
\leq a+2\eps)$ and for all $a\leq b$ 
$M'(a+\eps \leq b+\eps) \circ \varphi_a = \varphi_b \circ M(a\leq b)$ 
and
$M(a+\eps \leq b+\eps) \circ \psi_a = \psi_b \circ M'(a\leq b)$.
For persistence
modules $M$ and $M'$ define the \emph{interleaving distance} between
$M$ and $M'$ by
\[
d_I(M,M) = \inf( \eps \st M \text{ and } M' \text{ are }
\eps\text{-interleaved}).
\]

\begin{theorem} \label{thm:interleaving}
  $\Lambda_{\infty}(M,M') \leq d_I(M,M')$.
\end{theorem}

\begin{proof}
  Assume that $M$ and $M'$ are $\eps$-interleaved.  Then for $t \in
  \R$ and $m \geq \eps$, the map $M(t-m \leq t+m)$ factors through the
  map $M'(t-m+\eps \leq t+m-\eps)$.  So by Lemma~\ref{lem:rank},
  $\beta^{t-m+\eps,t+m-\eps}(M') \geq \beta^{t-m,t+m}(M)$. Thus by
  Definition~\ref{def:landscape}, $\lambda'(k,t) \geq \lambda(k,t) - \eps$
  for all $k\geq 1$.  It follows that
  $\norm{\lambda-\lambda'}_{\infty} \leq \eps$.
\end{proof}

\begin{proof}[Proof of Theorem~\ref{thm:inf-stability}]
  Combining Theorem~\ref{thm:interleaving} with the stability theorem
  of \citet{bubenikScott:1}, we have
  $\Lambda_{\infty}(M(f),M(g)) \leq d_I(M(f),M(g)) \leq \norm{f-g}_{\infty}$.
\end{proof}

\begin{proof}[Proof of Theorem~\ref{thm:inf-bound}]
  For a persistence diagram $D$, consider the persistence module given
  by the corresponding sum of interval
  modules \citep{csgo:persistenceModules}, $M(D) = \oplus_{(a,b)\in
    \hat  {D}}
  \mathbb{I}(a,b)$. Combining Theorem~\ref{thm:interleaving}
  with Theorem 4.9 of \citet{csgo:persistenceModules} we have
  $\Lambda_{\infty}(M(D),M(D')) \leq d_I(M(D),M(D')) \leq W_{\infty}(D,D')$.
\end{proof}
  
\begin{proof}[Proof of Theorem~\ref{thm:stability}]
  Let $\varphi: D \isomto D'$ with $\varphi(x_j) = x'_j$.
  Let $\lambda = \lambda(D)$ and $\lambda' = \lambda(D')$. So
  $\Lambda_p(D,D')^p = \norm{\lambda-\lambda'}_p^p$.
  \begin{align*}
    \norm{\lambda-\lambda'}_p^p &= \int
    \abs{\lambda(k,t)-\lambda'(k,t)}^p \\
    &= \sum_{k=1}^n \int \abs{\lambda_k(t)-\lambda'_k(t)}^p \, dt\\
    &= \int \sum_{k=1}^n \abs{\lambda_k(t)-\lambda'_k(t)}^p \, dt
  \end{align*}
Fix $t$. Let $u_j(t) = \lambda(\{x_j\})(1,t)$ and $v_j(t) =
\lambda(\{x'_j\})(1,t)$.
For each $t$, let $u_{(1)}(t) \leq \cdots \leq u_{(n)}(t)$ denote an
ordering of $u_1(t), \ldots, u_n(t)$ and define $v_{(k)}(t)$ for $1
\leq k \leq n$ similarly.
Then $u_{(k)}(t) = \lambda_k(t)$ and $v_{(k)}(t) = \lambda'_k(t)$ (see Figure~\ref{fig:pl}).
We obtain the result from the following where the two inequalities
are proven in Lemmata \ref{lem:order} and \ref{lem:norm-f}.
\begin{align*}
  \norm{\lambda-\lambda'}_p^p &= \int \sum_{k=1}^n \abs{ u_{(k)}(t) -
    v_{(k)}(t) }^p \, dt\\
  &\leq \int \sum_{k=1}^n \abs{u_k(t)-v_k(t)}^p \, dt\\
  &= \sum_{j=1}^n \int \abs{u_j(t)-v_j(t)}^p \, dt\\
  &\leq \sum_{j=1}^n \ell_j \eps_j^p + \frac{2}{p+1} \sum_{j=1}^n
  \eps_j^{p+1}. %\qedhere
\end{align*}
\end{proof}

\begin{lemma} \label{lem:order}
  Let $u_1,\ldots,u_n \in \R$ and $v_1,\ldots,v_n \in \R$. Order 
  them $u_{(1)} \leq \cdots \leq u_{(n)}$ and $v_{(1)} \leq \cdots
  \leq v_{(n)}$. Then 
  \begin{equation*}
    \sum_{j=1}^n \abs{u_{(j)}-v_{(j)}}^p \leq \sum_{j=1}^n\abs{u_j-v_j}^p.
  \end{equation*}
\end{lemma}

\begin{proof}
  Assume $u_1 < \cdots < u_n$, $v_1 < \cdots < v_n$, and $p \geq 1$.
  Let $u$ and $v$ denote $(u_1,\ldots,u_n)$ and $(v_1,\ldots,v_n)$.
  Let $\Sigma_n$ denote the symmetric group on $n$ letters and let
  $f_n: \Sigma_n \to \R$ be given by $f_n(\sigma) = \sum_{j=1}^n
  \abs{u_j - v_{\sigma(j)}}^p$. We will prove by induction that if
  $f_n(\sigma)$ is minimal then $\sigma$ is the identity, which we
  denote by $1$.

  For $n=1$ this is trivial. For $n=2$ assume without loss of
  generality that $u_1=0$, $u_2=1$ and $0 \leq v_1 < v_2$. 
  Let $1$ and $\tau$ denote the elements of $\Sigma_2$. 
  Then $f(1) = v_1^p + \abs{1-v_2}^p$ and $f(\tau) = v_2^p +
  \abs{1-v_1}^p$. Notice that $f(1) < f(\tau)$ if and only if $v_1^p -
  \abs{1-v_1}^p < v_2^p - \abs{1-v_2}^p$.
  The result follows from checking that $g(x) = x^p -\abs{1-x}^p$ is
  an increasing function for $x \geq 0$.

  Now assume that the statement is true for some $n \geq 2$.
%  Let $f_n^i(\sigma) = \sum_{j=1,j\neq i}^n
%  \abs{u_j-v_{\sigma(j)}}^p$, for $1 \leq i \leq n$.
  Assume that $f_{n+1}(\sigma^*)$ is minimal.
  % It follows that for $1 \leq i \leq n+1$,
  % $f_{n+1}^i(\sigma^*\mid_{1,\ldots,\hat{i},\ldots,n+1})$ is minimal.
  Fix $1 \leq i \leq n+1$.
  % Consider $u_1 < \cdots < u_{n+1}$ and $v_1 < \cdots < v_{n+1}$ with
  % $u_i$ and $v_{\sigma^*(i)}$ removed.
  Let $u' = (u_1,\ldots, \hat{u}_i,\ldots,u_{n+1})$ and
  $v'=(v_1,\ldots,\hat{v}_{\sigma^*(i)},\ldots,v_{n+1})$, where
  $\hat{\cdot}$ denotes omission.
  Since $f_{n+1}(\sigma^*)$ is minimal for $u$ and $v$,
  it follows that $\sum_{j=1, j\neq i}^n \abs{u_j-v_{\sigma^*(j)}}$ is
  minimal for $u'$ and $v'$.
  By the induction hypothesis, for $1 \leq j < k \leq n+1$ and $j,k\neq
  i$, $\sigma^*(j)<\sigma^*(k)$.
  Therefore $\sigma^*=1$.
  Thus, by induction, the statement is true for all $n$.

  Hence $\sum_{j=1}^n\abs{u_{(j)}-v_{(j)}}^p \leq
  \sum_{j=1}^n\abs{u_j-v_j}^p$ if $u_{(1)}<\cdots <u_{(n)}$ and
  $v_{(1)}<\cdots < v_{(n)}$.
  The statement in the lemma follows by continuity.
\end{proof}

\begin{lemma} \label{lem:norm-f}
  Let $x = (b,d)$ and $x'=(b',d')$ where $b\leq d$ and $b'\leq d'$.
  Let $\ell = {d-b}$ and $\eps = \norm{x-x'}_{\infty}$. Then
  $\norm{\lambda(\{x\})-\lambda(\{x'\})}_p^p \leq \ell \eps^p +
  \frac{2}{p+1} \eps^{p+1}$.
\end{lemma}

\begin{proof}
  Let $\lambda = \lambda(\{x\})$ and $\lambda' = \lambda(\{x'\})$.
  First $\lambda_k = \lambda'_k = 0$ for $k >1$; so
  $\norm{\lambda-\lambda'}_p = \norm{\lambda_1-\lambda'_1}_p$.
  Second $\lambda_1(t) = (h-\abs{t-m})_+$, where $h = \frac{d-b}{2}$,
  $m=\frac{b+d}{2}$, and $y_+ = \max(y,0)$, and similarly for
  $\lambda'_1$ (see Figure~\ref{fig:pl}).
  
  Fix $x$ and $\eps$.
  As $x'$ moves along the square $\norm{x-x'}_{\infty}=\eps$,
  $\norm{\lambda_1-\lambda'_1}_p^p$ has a maximum if
  $x'=(a-\eps,b+\eps)$.
  In this case $\norm{\lambda_1-\lambda'_1}_p^p = 2\int_0^h\eps^p\,dt+2\int_0^{\eps}t^p\,dt=\ell\eps^p+\frac{2}{p+1}\eps^{p+1}$.
\end{proof}

\begin{proof}[Proof of Corollary~\ref{cor:Wp}]
  Let $\varphi:D \isomto D'$ be a minimizer for $W_p(D,D')$, with
  corresponding $\{\eps_j\}$.
  Assume that $W_p(D,D') \leq 1$.
  Then $W_p(D,D')^p = \sum_{j=1}^n \eps_j^p \leq 1$. So for $1 \leq j
  \leq n$, $\eps_j \leq 1$.
  Combining this with Theorem~\ref{thm:stability}, we have that
  \begin{equation}
    \Lambda_p(D,D')^p \leq \sum_{j=1}^n\left(\ell_j +
      \frac{2}{p+1}\right) \eps_j^p.
  \end{equation}
  Since $W_{\infty}(D,\emptyset) = \max \frac{1}{2}\ell_j$, 
  $\ell_j \leq 2 \: W_{\infty}(D,\emptyset)$. Hence
  \begin{equation}
    \Lambda_p(D,D')^p \leq 2 \left( W_{\infty}(D,\emptyset) +
      \frac{1}{p+1} \right) W_p(D,D')^p.
  \end{equation}

  Therefore $W_p(D,D')^p \geq 1$ or $W_p(D,D')^p \geq
  \frac{1}{2}\left[W_{\infty}(D,\emptyset)+\frac{1}{p+1}\right]^{-1}
  \Lambda_p(D,D')^p$.
  The statement of the corollary follows.
\end{proof}

Theorem~\ref{thm:landscape-stability} follows from the following
corollary to Theorem~\ref{thm:stability} which is of independent
interest.

\begin{corollary} \label{cor:landscape-stability}
%  For persistence diagrams $D$ and $D'$,
  Let $p \geq k \geq 1$. Then 
  \begin{multline*}
  \Lambda_p(D,D')^p \leq W_{\infty}(D,D')^{p-k} \biggl[
    W_{\infty}(D,\emptyset)(\Pers_k(D)+\Pers_k(D')) + \\
    \frac{1}{p+1}(\Pers_{k+1}(D)+\Pers_{k+1}(D')) \biggr]
  \end{multline*}
\end{corollary}

\begin{proof}
  Let $\varphi$ be a minimizer for $W_{\infty}(D,D')$
  with corresponding $\{\eps_j\}$.
  If $\eps_j > \frac{\ell_j}{2} + \frac{\ell'_j}{2}$ then modify
  $\varphi$ to pair $x_j = (b_j,d_j)$ with $\bar{x}_j =
  (\frac{b_j+d_j}{2},\frac{b_j+d_j}{2})$ and similarly for $x'_j$.
  Note that $\norm{x_j-\bar{x}_j}_{\infty} = \frac{\ell_j}{2}$ and
  $\norm{x'_j-\bar{x'}_j}_{\infty} = \frac{\ell'_j}{2}$, so $\varphi$
  is still a minimizer for $W_{\infty}(D,D')$.

  Recall that for all $j$, $\ell_j \leq 2\, W_{\infty}(D,\emptyset)$.
  Since $\varphi$ is a minimizer for $W_{\infty}(D,D')$,
  for all $j$, $\eps_j \leq W_{\infty}(D,D')$.
  So applying our choice of $\varphi$ to Theorem~\ref{thm:stability}
  we have,
  \begin{equation*}
  \Lambda_p(D,D')^p \leq W_{\infty}(D,D')^{p-k} \left[
    2\, W_{\infty}(D,\emptyset) \sum_{j=1}^n \eps_j^k +
    \frac{2}{p+1}\sum_{j=1}^n \eps_j^{k+1} \right].
  \end{equation*}
  Now $\eps_j^q \leq \left( \frac{\ell_j}{2} + \frac{\ell'_j}{2}
  \right)^q \leq \frac{1}{2} \left( (\ell_j)^q + (\ell'_j)^q \right)$
  for $q \geq 1$, where the right hand side follows by the convexity
  of $\alpha(x) = x^q$ for $q \geq 1$.
  Thus $\sum_{j=1}^n \eps_j^q \leq \frac{1}{2} ( \Pers_q(D) +
  \Pers_q(D') )$ for $q \geq 1$. The result follows.
\end{proof}

\begin{proof}[Proof of Theorem~\ref{thm:landscape-stability}]
  Theorem~\ref{thm:landscape-stability} follows from
  Corollary~\ref{cor:landscape-stability} by the following two
  observations.  First, by the stability theorem of
  \citet{cseh:stability}, $W_{\infty}(D(f),D(g)) \leq
  \norm{f-g}_{\infty}$ and $W_{\infty}(D(f),\emptyset) \leq
  \norm{f}_{\infty}$.  Second, if $\Pers_q(D(f)) \leq C_{X,q}$ for all
  tame Lipschitz functions $f:X \to \R$ with $\Lip(f) \leq 1$, then
  for general tame Lipschitz functions, $\Pers_q(D(f)) \leq C_{X,q}
  \Lip(f)^q$.
\end{proof}

% \vskip 0.2in
% \bibliographystyle{plainnat}
% \bibliography{my}

\end{document}